\documentclass[pdflatex]{sn-jnl}%
\usepackage{amssymb}
\usepackage{xcolor}
\usepackage{tikz}
\usepackage{graphicx}%
\usepackage{multirow}%
\usepackage{amsmath,amssymb,amsfonts}%
\usepackage{amsthm}%
\usepackage{mathrsfs}%
\usepackage[title]{appendix}%
\usepackage{xcolor}%
\usepackage{textcomp}%
\usepackage{manyfoot}%
\usepackage{booktabs}%
\usepackage{algorithm}%
\usepackage{algorithmicx}%
\usepackage{algpseudocode}%
\usepackage{listings}%

\theoremstyle{thmstyleone}%
\newtheorem{theorem}{Theorem}
\newtheorem{lemma}[theorem]{Lemma}
\newtheorem{corollary}[theorem]{Corollary}

\theoremstyle{thmstyletwo}%
\newtheorem{remark}{Remark}%

\theoremstyle{thmstylethree}%
\newtheorem{definition}{Definition}%

\DeclareMathOperator{\loc}{loc}
\DeclareRobustCommand\full  {\tikz[baseline=-0.6ex]\draw[thick] (0,0)--(0.5,0);}
\DeclareRobustCommand\dotted{\tikz[baseline=-0.6ex]\draw[thick,dotted] (0,0)--(0.54,0);}
\DeclareRobustCommand\dashed{\tikz[baseline=-0.6ex]\draw[thick,dashed] (0,0)--(0.54,0);}

\DeclareRobustCommand\oline{%
\tikz[baseline=-0.6ex]{
  \draw[thick] (0,0)--(0.5,0);
  \filldraw (0.25,0) circle (1.5pt);
}}
\DeclareRobustCommand\chainn{\tikz[baseline=-0.6ex]\draw[thick,dash dot] (0,0)--(0.5,0);}
\usepackage{bbm}
\usepackage[utf8]{inputenc}  
\usepackage{amssymb,amsmath}
\usepackage{enumerate}
\usepackage[shortlabels]{enumitem}
\usepackage{mathtools}
\usepackage{graphicx}
\usepackage{url}
\usepackage{color}  
\definecolor{darkred}{rgb}{0.6,0,0}
\definecolor{darkgreen}{rgb}{0,0.5,0}
\definecolor{darkmagenta}{rgb}{0.5,0,0.5}
  
\usepackage[capitalize,nameinlink]{cleveref}

\makeatletter
\def\@cite#1#2{\textup{[{#1\if@tempswa , #2\fi}]}}
\makeatother

\DeclarePairedDelimiter\abs{\lvert}{\rvert}

   \DeclarePairedDelimiterX\Set[1]\{\}{%
      
      #1
   }

\newcommand{\R}{\mathbb R}

\newcommand{\dott}{\, \cdot\,}

 \usepackage[sort]{cite}
\usepackage{bigints}
\newcommand{\Z}{\mathbb{Z}}
\DeclareMathOperator{\lip}{Lip}
\DeclareMathOperator{\bv}{BV}
\newcommand{\sgn}{\mathop\mathrm{sgn}}
\newcommand{\norma}[1]{{\left\|#1\right\|}}

\renewcommand{\d}[1]{\mathinner{\mathrm{d}{#1}}}
\DeclareMathOperator{\TV}{TV}

\newcommand{\D}{\Delta}
\newcommand{\modulo}[1]{{\left|#1\right|}}

\numberwithin{equation}{section}     
\allowdisplaybreaks

\begin{document}
\title[]{Systems of Nonlocal Conservation Laws with Memory and Their Zero Retention Limit}
\author*[1]{\fnm{Aekta} \sur{Aggarwal}}\email{aektaaggarwal@iimidr.ac.in}

\author[2]{\fnm{Ganesh} \sur{Vaidya}}\email{vaidyaganesh@iisc.ac.in}

\affil*[1]{\orgdiv{Operations Management and Quantitative Techniques}, 
\orgname{Indian Institute of Management Indore}, 
\orgaddress{\street{Prabandh Shikhar, Rau--Pithampur Road}, 
\city{Indore}, \postcode{453556}, \state{Madhya Pradesh}, \country{India}}}

\affil[2]{\orgdiv{Department of Mathematics}, 
\orgname{Indian Institute of Science}, 
\orgaddress{\city{Bangalore}, \postcode{560012},  \state{Karnataka}, \country{India}}}
\date{\today}


 \abstract{
We study the entropy solution for a class of systems of nonlocal conservation laws in which the convective flux is convoluted with a kernel in both spatial and temporal variables. This formulation models the flux dependence on the solution within its spatial neighborhood (nonlocal in space) as well as on prior states in time (nonlocal in time), thereby incorporating memory effects. In addition, employing a convergent finite volume approximation, the existence of the entropy solution is discussed. The uniqueness
of such entropy solutions is also established.

In addition, we analyze the asymptotic behavior of the solutions as the
support of the temporal convolution kernel shrinks, demonstrating the “memory-
to-memoryless” effect and convergence to the entropy solution of the corresponding nonlocal conservation law without memory (i.e., nonlocal only in
space). Convergence rate estimates are derived. In addition, the proposed numerical approximations are shown to be asymptotically compatible with this
passage to the memoryless limit by deriving the corresponding asymptotic
convergence rate estimates. The analysis is carried out in a very general set-
ting, without imposing any geometric restrictions such as the convexity of the
spatial and temporal convolution kernels, unlike the existing literature on the
asymptotic analysis of nonlocal-in-space only conservation laws. To the best
of our knowledge, this provides the first convergence and asymptotic analysis
for finite volume schemes applied to nonlocal conservation laws with memory. Numerical experiments are included to illustrate the theory.
}



\maketitle



\section{Introduction}
Many real-world phenomena exhibit dynamics in which the state at a given time depends not only on the present configuration but also on the past history of the system. In such settings, the notion of “memory” 
captures the influence of previous states on the current dynamics. Such hereditary structures appear in industrial applications such as damping phenomena in elastic media~\cite{CHR2007,c2008}, viscoelasticity~\cite{Dafermos1970}, models involving fractional-in-time behavior \cite{Podlubny1999},  gas transport models in coal seams involving adsorption and diffusion in meso and micropores~\cite{Clarkson1999,Shi2003}, and transport models in subsurface and hyporheic zones~\cite{Haggerty1995,Gooseff2003,Haggerty2002}. A prototypical scalar conservation law modeling such events,  reads as follows:
\begin{equation}\label{eq:memory_general}
\partial_t u
+ \partial_x \Big( \, f(u,\int_0^t u(\tau,x)\Gamma(t-\tau))\, d\tau \Big)
=0, \quad (t,x) \in Q_T:= (0,T)\times\R,
\end{equation}
where $\Gamma$ is a temporal memory kernel.
Despite their applications in modeling a wide range of phenomena, the well-posedness theory for such models remains largely incomplete. One key difficulty is that the temporal convolution term often lacks spatial smoothness, rendering the flux discontinuous as a function of the spatial variable. The presence of temporal convolution destroys the semigroup property in the usual sense and gives rise to a new mathematical object, unlike those arising in local or purely spatially nonlocal conservation laws. 
Consequently, classical methods for standard conservation laws do not directly apply. Some studies ~\cite{c2008,N2023, DHSS2023, P2014,liu2020, Dafermos1970, D1987, CC2007,DAF2012,CHR2007} deal with some specific cases  exploiting structural properties of particular PDEs to study the wellposedness, under some specialized assumptions on the flux. 

However, the case of ``nonlocal space only" conservation laws, \begin{align}
    \label{nls1}
  \partial_t u +\partial_x \Big(f(u,\int_{\R} u(t,\xi)\mu(x-\xi))\, d\xi \Big)&=0, \quad (t,x) \in Q_T,
\end{align} with $\mu$ being a spatial horizon kernel, has gained significant interest in the recent decade,  due to two reasons: their widespread use in  modeling of applications where the flux at a given point may depend 
not only on the local state but also on averaged quantities over a finite 
interaction horizon, for example,  crowds \cite{CGL2012,ACG2015}, traffic \cite {BG2016,BHL2023,FGKP2022,AHV2023_1,AHV2024}, sedimentation
models \cite{BBKT2011}, {laser technology \cite{CM2015}},  granular material dynamics \cite{AS2012}, conveyor belt dynamics \cite{GHS+2014},  opinion formation~\cite{ANT2007}, sedimentation
models~\cite{BBKT2011} and  granular material dynamics~\cite{AS2012} etc, and their analytical/mathematical vicinity to local nonlinear conservation laws. Its well-posedness has been well studied in literature, a non-exhaustive list being \cite{CK24,FGR2021,BFK2022,AV2023,FCV2023,KP2021,CG2019, KLS2018,ANT2007,AS2012,BHL2023, ACT2015, AG2016,BG2016,FGKP2022,CGL2012,AHV2023_1,AHV2024,AHV2023, ACG2015,BBKT2011,GHS+2014,CG2023,CM2015} etc. 

In contrast, the present article combines these two frameworks by considering convolution kernels that depend on both spatial and temporal variables. 
In particular, we study the initial-value problem (IVP) for the coupled system of $N$ nonlocal hyperbolic conservation laws of the above type, where for every $k\in \mathcal{N}:=\{1, \ldots, N\}$,  the $k^{\rm th}$ equation is given by
\begin{align}
    \label{nlm}
  \partial_t U^{k} +\partial_x \Big(f^k(U^k) \nu^k((\boldsymbol{\Theta} \circledast  \boldsymbol{U})^k)\Big) &=0, \quad (t,x) \in Q_T,\\
  \label{init}
  U^{k}(0,x)&=U_0^{k}(x), \quad x \in \R.
 \end{align} Here, $T$ is the final time and the unknown is $\boldsymbol{U}=(U^{k})_{k\in\mathcal{N}}:[0,\infty)\times\mathbb{R}\to\mathbb{R}^N.$ 
Further, for every $j,k\in \mathcal{N}$, ${U}^k_0\in (L^1 \cap \bv) (\R;[0,1]), \Theta^{j,k}(t,x):=\mu^{j,k}(x) \Gamma^{j,k}(t)$, with $\boldsymbol{\mu}$ and $\boldsymbol{\Gamma}$ being smooth $N\times N$ matrices, and $(\boldsymbol{\Theta}\circledast \boldsymbol{U})^k:= (\Theta^{i,k}*U^i)_{i\in\mathcal{N}},$ where for every $(t,x)\in \overline{Q}_T$, \begin{align}\label{mc}
\begin{split}
   (\Theta^{i,k}*U^i)(t,x)&:=\displaystyle\int_0^t\int_{\R} U^i(\tau,\xi)\mu^{j,k}(x-\xi)\Gamma^{i,k}(t-\tau) \d \tau \d \xi, i\in\mathcal{N}.
\end{split}
\end{align} Additionally,
\begin{enumerate}[(\textbf{H\arabic*})]
\item \label{H1A}$f^k \in  \lip(\R)$  with $ f^k(0)=0=f^k(1)$.
 \item \label{H2A}$\nu^k \in (C^2 \cap \bv \cap \, W^{2,\infty}) (\R^N,\R)$;
 \item \label{H3A}
 $\Theta^{j,k}(t,x):=\mu^{j,k}(x) \Gamma^{j,k}(t)$, with $\boldsymbol{\mu}$ and $\boldsymbol{\Gamma}$ being smooth $N\times N$ matrices. The space kernel ${\mu}^{j,k} \in C^2(\R) \cap W^{2,\infty}(\R)$. The time kernel $\Gamma^{j,k}\in C^2([0,\infty);\R^+) \cap W^{2,\infty}([0,\infty);\R^+).$ 
 \end{enumerate} Modeling-wise, it is to be noted that the system \eqref{nlm}--\eqref{init} is coupled through the nonlocal operator
$
\boldsymbol{\Theta}\circledast\boldsymbol{U}$,
which introduces both \emph{spatial averaging} and \emph{temporal memory effects} into the convective flux of \eqref{nlm}.
More precisely, for every $k\in\mathcal{N}$, the flux at a spacetime point $(t_0,x_0)$
depends not only on the instantaneous density $U^k(t_0,x_0)$, but also on
\begin{enumerate}[(i)]
\item a weighted average of the states $U^j(t_0,\cdot)$ in a neighborhood of $x_0$,
determined by the spatial kernel $\mu^{j,k},j,k,\in\mathcal{N}$;
\item a cumulative influence of past states $U^j(\tau,x_0)$, $\tau<t_0$,
modulated by the time kernel $\Gamma^{j,k},j,k,\in\mathcal{N}$.
\end{enumerate}
Consequently, the characteristic speed of the $k^{\mathrm{th}}$ equation at $(t_0,x_0)$
reflects a history-dependent interaction between all components of the system,
rather than being determined solely by local instantaneous data. In practical applications,  $\Gamma^{j,k}$ is typically taken as a decreasing function of time, depicting fading memory and assigning smaller weights to older information. However, the results of this article remain valid for a broader class of $\Gamma^{j,k}$ satisfying \ref{H3A}.
 
 We will show that although a general well-posedness theory for time-nonlocal conservation laws remains largely undeveloped, the present framework enables us to systematically adapt techniques from the theory of purely spatial nonlocal conservation laws to establish well-posedness for a substantially broader class of flux functions incorporating both spatial and temporal nonlocal effects, invoking the  regularity of the convective flux with respect to the space variable $x$.
{
Using relaxation schemes, \eqref{nlm}-\eqref{init} has been recently studied for $N=1$ in \cite{DHSS2023} for a specific choice of $\Theta$ and linear $f$}.
That said, the question of existence as well as uniqueness of the entropy solutions for the system  \eqref{nlm}-\eqref{init} with a general $\boldsymbol{\Theta}$ having temporal dependence and(or) for $N>1$, remains unexplored and unsettled as of now, which is dealt with in the first part of this article. {\color{black}Several further interesting questions} regarding the entropy solution of \eqref{nlm}-\eqref{init} are: 
\begin{enumerate}[(\textbf{Q\arabic*})]
\item \label{a} Does it converge to the entropy solution of its ``nonlocal-space'' counterpart \begin{align}
    \label{nls}
  \partial_t U^{k} +\partial_x \Big(f^k(U^k) \nu^k((\boldsymbol{\mu} \circledast  \boldsymbol{U})^k)\Big) &=0,\quad (t,x) \in Q_T,\\
\label{sc}
(\mu^{j,k}*U^j)(t,x)&:=\displaystyle\int_{\R} U^j(t,\xi)\mu^{j,k}(x-\xi)  \d \xi,
\end{align}
as the radius of the temporal convolution kernel $\boldsymbol{\Gamma}$ tends to zero?  \item \label{b} Does it converge to the entropy solution of its ``nonlocal-time'' counterpart 
as the radius of the spatial convolution kernel $\boldsymbol{\mu}$ tends to zero?  
\item \label{c} Does it converge to the entropy solution of its local counterpart
\begin{align}\label{local}
\partial_t U^{k} + \partial_x \Big(f^k(U^k)\, \nu^k(\boldsymbol{U})\Big) = 0, \quad (t,x) \in Q_T,
\end{align}
as the radii of the temporal and spatial convolution kernels $\boldsymbol{\Gamma}$ and $\boldsymbol{\mu}$ tend to zero? 
\item \label{d} Do there exist finite volume approximations for \eqref{nlm}–\eqref{init} that are asymptotically compatible with any of the limiting passages described above? If so, can one quantify the corresponding rates of this asymptotic convergence?
\end{enumerate}
In this article, we would obtain uniform spatial $\bv$ bounds on $\boldsymbol{U}$ independent of $\|\Gamma^{j,k}\|_{L^\infty(\R)}$ in \S\ref{num}, using which, we would answer \ref{a} and corresponding \ref{d} affirmatively, thereby confirming the passage from the ``memory'' model to its ``memoryless'' counterpart in \S\ref{NLL}. 

The passage from the ``nonlocal space-time'' system \eqref{nlm}-\eqref{init} to the ``nonlocal-time'' setting \ref{b} is not addressed here since this yields a system with a flux that is local and is discontinuous in space; 
the well-posedness of such systems is largely open.

Lastly, the answer to \ref{c} remains largely open, with the exception of \cite{DHSS2023}, which addresses this issue for $N=1$ with a specific choice of $\Theta,$ and a linear flux $f$ via a relaxation scheme. 
In general, questions \ref{c} and \ref{d} with  $N=1, \boldsymbol{\Theta}=\boldsymbol{\mu}$ and for the specific choice for a linear $f,$ and $\mu$ a decreasing kernel have been investigated in recent years 
see, for instance, \cite{BS2020,BS2021,CCDKP2022,GNAL2021,CGES2021,CCS2019,FGKP2022,KP2019,CAL2020,FKG2018}, and  \cite{DH2024,NH2025} respectively. The case of a general $\mu$ and nonlinear $f$ remains open, to the best of our knowledge. {\color{black}However, for systems ($N>1$), questions \ref{c} and \ref{d} remain largely undeveloped and highly nontrivial. While recent works~\cite{AHV2023,AHV2023_1} prove that the nonlocal coupled system \eqref{nlm}–\eqref{init} is well posed for arbitrary $\bv$ initial data, the corresponding local system \eqref{local},\eqref{init} exhibits severe analytical difficulties and is, in general, ill-behaved with very limited well-posedness results:  existence for any $N$ with sufficiently small $\bv$ data~\cite{B2000,HR2015}, and  existence for $N=2$ with sufficiently small $L^{\infty}$ data~\cite{BMV2025}.}
 This substantial disparity makes question \ref{c} very delicate for $N>1$, 
and it remains an elusive open problem at the time of writing even for $\boldsymbol{\Theta}=\boldsymbol{\mu}$. Similarly, questions \ref{c} and, consequently, \ref{d} with  $\boldsymbol{\Theta}=\boldsymbol{\Gamma}$ for any $N\ge 1$ remains still unexplored.
Consequently, an affirmative answer to \ref{c}-\ref{d}, i.e., the complete passage for any general $\boldsymbol{\Theta}$ and for any nonzero integer $N$ from the ``nonlocal space-time'' system \eqref{nlm}-\eqref{init} to the local limit \eqref{local},\eqref{init} is difficult to expect and lies beyond the scope of this paper. 
The remainder of the paper is organized as follows. In \S\ref{uni}, we establish uniqueness and continuous dependence with respect to the initial data for \eqref{nlm}–\eqref{init}, based on stability estimates for local scalar conservation laws. In \S\ref{num}, we propose a first-order numerical scheme for the approximation of the initial value problem and prove convergence of the approximate solutions to the unique entropy solution, thereby establishing existence. Taken together, these results yield the well-posedness of the problem. In \S\ref{NLL}, we analyze the dynamics in the presence of fading memory, the memory-to-memoryless dynamics. In particular, we investigate the asymptotic behavior as the support of the temporal convolution kernel shrinks to zero. We prove that the unique entropy solutions of the space–time nonlocal conservation law converge to the entropy solution of the purely spatially nonlocal conservation law at the rate $\mathcal{O}(\sqrt{\delta})$, where $\delta$ denotes the radius of the time convolution. In addition, we show that the proposed finite volume approximations are asymptotically compatible with this memory-to-memoryless limit and derive error estimates of the order $\mathcal{O}(\sqrt{\delta} + \sqrt{\Delta x})$. Finally, in \S\ref{num1},  numerical experiments supporting the theory of the article are presented.
\section{Definitions and notation}\label{def}
We will be using the following notations:
\begin{enumerate}
    \item For $\boldsymbol{Z}:=(Z^k)_{k\in\mathcal{N}}\in \R^N,$ let $\norma{\boldsymbol{Z}}:=\displaystyle\sum\limits_{k\in\mathcal{N}}\abs{Z^k}$ denote the usual $1$-norm.
    \item 
$
\norma{\boldsymbol{\Theta}}_{(L^{\infty}(\overline{Q}_T))^{N^2}}:= \max\limits_{i,j\in\mathcal{N}} \norma{\Theta^{i,j}}_{L^{\infty}(\overline{Q}_T)}.
$
\item If $\boldsymbol{\mu}\in C^2(\R;\R^{N^2})=(\mu^{j,k})_{j,k\in\mathcal{N}}$, then $\boldsymbol{\mu}'=({\dot{\mu}}^{j,k})_{j,k\in\mathcal{N}}\in C^1(\R;\R^{N^2})$ and $\boldsymbol{\mu}''=({\ddot{\mu}}^{j,k})_{j,k\in\mathcal{N}} \in C(\R;\R^{N^2})$ denote the component-wise derivative and second derivative, respectively. 
   \item For $u: \overline{Q}_T\rightarrow \R,$ $\boldsymbol{U}:\overline{Q}_T \rightarrow \R^N$ 
  and $\tau>0$, 
\begin{align*}
|u|_{\lip_tL^1_x}&:=\sup_{0\leq t_1<t_2\leq T}\frac{\norma{u(t_1,\cdot)-u(t_2,\cdot)}_{L^1(\R)}}{|t_1-t_2|},\\
|\boldsymbol{U}|_{(\lip_tL^1_x)^N}&:=\max_{k\in\mathcal{N}}|U^k|_{\lip_tL^1_x}, \\  
|\boldsymbol{U}|_{(L^\infty_t\bv_x)^N}&:=\max_{k\in\mathcal{N}}\sup_{t\in[0,T]} TV(U^k(t,\dott)),\\  
\norma{\boldsymbol{U}}_{(L^{\infty}(\overline{Q}_T))^N}&:=\max_{k\in\mathcal{N}} \norma{U^k}_{L^{\infty}(\overline{Q}_T)}=1,\\ \gamma(\boldsymbol{U},\tau)&:= \max_{k\in\mathcal{N}}\sup_{\substack{
    \abs{t_1-t_2} \leq \tau\\  0\leq t_1\leq t_2\leq T }} \norma{U^k(t_1,\cdot)-U^k( t_2,\cdot)}_{L^1(\R)},\\ 
\norma{\boldsymbol{U}}_{(L^1(\overline{Q}_T))^N}&:=\sum\limits_{k\in\mathcal{N}}\norma{U^k}_{L^1(\overline{Q}_T)}.
\end{align*} 
\end{enumerate}
Since $f^k$ is nonlinear, there can be multiple weak solutions of \eqref{nlm}-\eqref{init}, like in a local hyperbolic conservation law. Hence, an entropy condition is required to single out the unique solution.
\begin{definition}\label{def:entropy}
    A function $\textbf{U} \in (C([0,T];L^1(\R;[0,1]))\cap L^{\infty}([0,T];\bv(\R)))^N$  is an entropy solution of \eqref{nlm}-\eqref{init} with initial data $\textbf{U}_0$
if for each $(k,\alpha)\in \mathcal{N}\times\R$, and for all non-negative $\phi\in C_c^{\infty}([0,T)\times \R),$
\begin{multline} \label{kruz2}
\int_{Q_T}\left|U^k(t,x)- \alpha\right|\phi_t(t,x)  \d x \d t+\int_{\R} \left|U_0^k(x)- \alpha\right|\phi(0,x) \d x  \\ 
+ \int_{Q_T}\sgn ({U}^k(t,x)-\alpha) \nu^k((\boldsymbol{\Theta}\circledast \boldsymbol{U})^k(t,x))
(f^k({U}^k(t,x))-f^k(\alpha))\phi_x(t,x) \d{x} \d{t}\\ 
 -\int_{Q_T} f^k(\alpha) (\sgn ({U}^k(t,x)-\alpha)) \partial_x\nu^k((\boldsymbol{\Theta}\circledast \boldsymbol{U})^k(t,x))\phi(t,x)\d{x} \d{t}\geq 0. 
\end{multline}
\end{definition}
\section{Uniqueness}\label{uni}
We now prove that any two entropy solutions of the IVP \eqref{nlm}-\eqref{init} are equal, more precisely, we have the following result:

\begin{theorem}[Uniqueness]
For any time $T>0,$ let $\boldsymbol{U,V}$ 
be the entropy solutions of the IVP for the system \eqref{nlm}-\eqref{init}
  with initial data $\boldsymbol{U}_0,\boldsymbol{V}_0$, respectively. Then, the following holds:
\begin{align*}
\begin{split}
\norma{\boldsymbol{U}(T,\dott)-\boldsymbol{V}(T,\dott)}_{(L^1(\R))^N}&\leq
\norma{\boldsymbol{U}_0-\boldsymbol{V}_0}_{(L^1(\R))^N}\\& +\mathcal{C} \sum\limits_{k\in\mathcal{N}}\int_0^T \norma{U^k(\tau,\cdot)-V^k(\tau,\cdot)}_{L^1(\R)} \d \tau,
\end{split}
\end{align*}   \text{where  }
 \begin{align*}
\mathcal{C}&= NT\abs{\boldsymbol{f}}_{(\operatorname{Lip}(\R))^N}|\boldsymbol{\nu}|_{(\operatorname{Lip}(\R^N))^N}\norma{\boldsymbol{\Theta}}_{{(L^{\infty}(\overline{Q}_T))}^{N^2}}|\boldsymbol{U}|_{(L^\infty_t\bv_x)^N}\\
&\quad+NT\abs{\boldsymbol{f}}_{(\operatorname{Lip}(\R))^N}{\norma{\boldsymbol{U}_0}_{{(L^1(\R))}^N}}\norma{{\color{black}\nabla }\boldsymbol{\nu}}_{(L^{\infty}(\R^N))^{N^2}}\norma{\boldsymbol{\Gamma}}_{(L^{\infty}(\R^{+}))^{N^2}}\norma{\boldsymbol{\mu}'}_{(L^{\infty}(\R))^{N^2}}\\
&\quad+NT\abs{\boldsymbol{f}}_{(\operatorname{Lip}(\R))^N}{\norma{\boldsymbol{U}_0}_{{(L^1(\R))}^N}}\norma{\boldsymbol{V}}_{(L^{1}(Q_T))^N} \norma{\boldsymbol{\Gamma}}_{(L^{\infty}(\R^{+}))^{N^2}}\norma{\boldsymbol{\mu}'}_{(L^{\infty}(\R))^{N^2}} \\&\qquad\times{\color{black}\norma{\operatorname{Hess}\boldsymbol{\nu}}}_{(L^{\infty}(\R^{N}))^{N^3}}\norma{\boldsymbol{\Theta}}_{{(L^{\infty}(\overline{Q}_T))}^{N^2}}.\end{align*}
 In particular, if $\boldsymbol{U}_0=\boldsymbol{V}_0,$ then $\boldsymbol{U}=\boldsymbol{V}$ a.e.~in $\overline{Q}_T$.
\end{theorem}
\begin{proof}
Let $(t,x)\in Q_T$. Since for each $k\in\mathcal{N}, U^k$ and $V^k$ are entropy solutions of \eqref{nlm}-\eqref{init},
using the continuous dependence estimates for conservation laws with smooth coefficients (see \cite[Theorem~1.3]{KR2003}) and following similar steps as in \cite[Theorem~4.1]{BBKT2011}, \cite[Theorem~2]{BG2016},  we have:
\begin{align*}
&\norma{U^k(T,\dott)-V^k(T,\dott)}_{L^1(\R)}\leq
\norma{U^k_0-V^k_0}_{L^1(\R)} + \abs{f^k}_{\lip(\R)}(\mathcal{I}_1^k+\mathcal{I}_2^k),
\end{align*}
where
\begin{align}\label{cont_dependennce}
\begin{split}
&\mathcal{I}_1^k:=\int_0^T\int_{\R}\abs{U^k_x(t,x)}  \abs{\nu^k((\boldsymbol{\Theta}\circledast \boldsymbol{U})^k(t,x))-\nu^k((\boldsymbol{\Theta}\circledast \boldsymbol{V})^k(t,x))} \d x \d\tau,\\ &\mathcal{I}_2^k:=\int_0^T\int_{\R} \abs{U^k(t,x)}  \abs{\partial_x(\nu^k((\boldsymbol{\Theta}\circledast \boldsymbol{U})^k(t,x)))-\partial_x(\nu^k((\boldsymbol{\Theta}\circledast \boldsymbol{V})^k(t,x)))} \d x \d\tau.
    \end{split}
\end{align}
 Note that  
 \begin{align} \label{U}
\begin{split}
&\abs{\nu^k((\boldsymbol{\Theta}\circledast \boldsymbol{U})^k(t,x))-\nu^k((\boldsymbol{\Theta}\circledast \boldsymbol{V})^k(t,x))}\\
&\le|\nu^k|_{\operatorname{Lip}(\R^N)}\norma{(\boldsymbol{\Theta}\circledast \boldsymbol{U})^k(t,x)-(\boldsymbol{\Theta}\circledast \boldsymbol{V})^k(t,x)}\\
&=|\nu^k|_{\operatorname{Lip}(\R^N)}\sum\limits_{j\in\mathcal{N}}\int_0^t\int_{\R}
\abs{U^j(\tau,\xi)-V^j(\tau,\xi)}\mu^{j,k}(x-\xi)\Gamma^{j,k}(t-\tau) \d \tau \d \xi\\
 &\le|\nu^k|_{\operatorname{Lip}(\R^N)}\norma{\boldsymbol{\Theta}}_{{(L^{\infty}(\overline{Q}_T))}^{N^2}}\sum\limits_{k\in\mathcal{N}}\int_0^t\norma{U^k(\tau,\cdot)-V^k(\tau,\cdot)}_{L^1(\R)} \d \tau.
 \end{split}
 \end{align}
 Also, for $\boldsymbol{Z}=\boldsymbol{U},\boldsymbol{V}$,
\begin{align*}
&\nabla\nu^k((\boldsymbol{\Theta}\circledast \boldsymbol{Z})^k(t,x))\dott \partial_x((\boldsymbol{\Theta}\circledast \boldsymbol{Z})^k)(t,x)\\&\quad=\sum\limits_{j\in\mathcal{N}}\partial_{j}\nu^{k}\!\left((\boldsymbol{\Theta}\circledast \boldsymbol{Z})^k(t,x)\right)\,
\int_{0}^{t}\!\int_{\mathbb{R}}
Z^{j}(\tau,\xi)\,\dot{\mu}^{j,k}(x-\xi) \d \xi \d \tau,\end{align*}
and hence,
\begin{align*}
&\nabla\nu^k((\boldsymbol{\Theta}\circledast \boldsymbol{U})^k(t,x))\dott \partial_x((\boldsymbol{\Theta}\circledast \boldsymbol{U})^k)(t,x)-\nabla\nu^k((\boldsymbol{\Theta}\circledast \boldsymbol{V})^k(t,x))\dott \partial_x((\boldsymbol{\Theta}\circledast \boldsymbol{V})^k)(t,x)\\
&=\sum\limits_{j\in\mathcal{N}}\partial_{j}\nu^{k}\!\left((\boldsymbol{\Theta}\circledast \boldsymbol{U})^k(t,x)\right)\,
\int_{0}^{t}\!\int_{\mathbb{R}}
(U^{j}(\tau,\xi)-V^{j}(\tau,\xi))\,\dot{\mu}^{j,k}(x-\xi)\,
\Gamma^{j,k}(t-\tau)\,d\xi\,d\tau\\
&\quad+\sum\limits_{j\in\mathcal{N}}(\partial_{j}\nu^{k}\!\left((\boldsymbol{\Theta}\circledast \boldsymbol{U})^k(t,x)\right)-\partial_{j}\nu^{k}\!\left((\boldsymbol{\Theta}\circledast \boldsymbol{V})^k(t,x)\right)\,
\int_0^t \int_{\mathbb{R}}V^{j}(\tau,\xi))\,\dot{\mu}^{j,k}(x-\xi)\,
\Gamma^{j,k}(t-\tau)\,d\xi\,d\tau\\
&=\mathcal{I}_{21}^k+\mathcal{I}_{22}^k.
\end{align*}
Clearly, with $\mathcal{C}_2:=\norma{{\color{black}\nabla }\boldsymbol{\nu}}_{(L^{\infty}(\R^N))^{N^2}}\norma{\boldsymbol{\Gamma}}_{(L^{\infty}(\R^{+}))^{N^2}}\norma{\boldsymbol{\mu}'}_{(L^{\infty}(\R))^{N^2}}$, \begin{align}
\label{I21}
\mathcal{I}_{21}^k
&\le \mathcal{C}_2\sum\limits_{k\in\mathcal{N}}\int_0^T\norma{U^k(\tau,\cdot)-V^k(\tau,\cdot)}_{L^1(\R)} \d \tau.
\end{align}
Further, using \eqref{U} and a similar estimate with $\partial_j\nu^k,$ for every $j\in\mathcal{N}$, we get:
\begin{align*}
&\abs{\partial_{j}\nu^{k}\!\left((\boldsymbol{\Theta}\circledast \boldsymbol{U})^k(t,x)\right)-\partial_{j}\nu^{k}\!\left((\boldsymbol{\Theta}\circledast \boldsymbol{V})^k(t,x)\right)}\\&\le \norma{\operatorname{Hess}\nu^k}_{(L^{\infty}(\R^{N}))^{N^2}}\norma{\boldsymbol{\Theta}}_{{(L^{\infty}(\overline{Q}_T))}^{N^2}}\sum\limits_{k\in\mathcal{N}}\int_0^T\norma{U^k(\tau,\cdot)-V^k(\tau,\cdot)}_{L^1(\R)} \d \tau,
\end{align*} and hence with, $$\mathcal{C}_3:=\norma{\boldsymbol{V}}_{(L^{1}(Q_T))^N} \norma{\boldsymbol{\Gamma}}_{(L^{\infty}(\R^{+}))^{N^2}}\norma{\boldsymbol{\mu}'}_{(L^{\infty}(\R))^{N^2}} {\color{black}\norma{\operatorname{Hess}\boldsymbol{\nu}}}_{(L^{\infty}(\R^{N}))^{N^3}}\norma{\boldsymbol{\Theta}}_{{(L^{\infty}(\overline{Q}_T))}^{N^2}},$$ \begin{align}\label{I22}
\mathcal{I}_{22}^k
&\le \mathcal{C}_3\sum\limits_{k\in\mathcal{N}}\int_0^T\norma{U^k(\tau,\cdot)-V^k(\tau,\cdot)}_{L^1(\R)} \d \tau.
\end{align} 
 Consequently, using \eqref{U}-\eqref{I22}, we have
\begin{align} \label{U1}
\begin{split}
\mathcal{I}_1^k
&\le\mathcal{C}_1\sum\limits_{k\in\mathcal{N}}\int_0^T\norma{U^k(\tau,\cdot)-V^k(\tau,\cdot)}_{L^1(\R)} \d \tau,\\
\mathcal{I}_2^k& \leq \int_{0}^{T}\int_{\R}\abs{U^k(\tau,x)} (\mathcal{I}_{21}^k+\mathcal{I}_{22}^k) \d x \d \tau
 \leq  \mathcal{C}_4\sum\limits_{k\in\mathcal{N}}\int_0^T\norma{U^k(\tau,\cdot)-V^k(\tau,\cdot)}_{L^1(\R)} \d \tau.
\end{split},
\end{align}
where $\mathcal{C}_1:=T |\boldsymbol{\nu}|_{(\operatorname{Lip}(\R^N))^N}\norma{\boldsymbol{\Theta}}_{{(L^{\infty}(\overline{Q}_T))}^{N^2}}|\boldsymbol{U}|_{(L^\infty_t\bv_x)^N}$, and $\mathcal{C}_4:=T{\norma{\boldsymbol{U}_0}_{{(L^1(\R))}^N}}(\mathcal{C}_2+\mathcal{C}_3).$
Finally, inserting the above estimates in \eqref{cont_dependennce} and summing over $k\in\mathcal{N}$, we have the result.
\end{proof}

\section{Numerical approximations and their convergence}\label{num}
For $\Delta x, \Delta t>0,$ and $\lambda:=\Delta t/\Delta x,$ consider equidistant spatial grid points $x_i:=i\Delta x$ for $i\in \Z,$ and let $\chi_i(x)$ denote the indicator function of $C_i:=[x_{i-1/2}, x_{i+1/2})$, where $x_{i+1/2}=\frac{1}{2}(x_i+x_{i+1})$. Further, let $t^n:=n\Delta t$ 
for integers in $\mathcal{N}_T:=\{0, \ldots, N_T\}$, such that $T=N_T \D t$ denote the temporal grid points, and let
$\chi^n(t)$ denote the indicator function of $C^{n}:=[t^n,t^{n+1})$. For every $k\in\mathcal{N}$, we approximate the initial data \eqref{init}, according to:
\begin{equation*}
U^{k,\Delta}_0(x):=\sum\limits_{i\in\Z}\chi_i(x)U^{k,0}_i\quad \mbox{where }U^{k,0}_i=\int_{C_i}U_{0}^k(x)\d x, \quad i\in \Z,
\end{equation*}
and define a piecewise constant approximate solution $U^{k,\Delta}$ to~\eqref{nlm}-\eqref{init}
by:
$$
  U^{k,\Delta} (t,x) =  U^{k,n}_{i}
  \mbox{ for } 
(t,x)\in {C}^{n}\times C_i, (i,n)\in \Z\times\mathcal{N}_T. 
$$
For every $(i,k,n)\in \Z\times\mathcal{N}\times\mathcal{N}_T$, $U^{k,n}_{i}$ is defined via the following marching formula:
\begin{align}
\begin{split}U^{k,n}_i&=H^k(\nu^k(\boldsymbol{c}^{k,n-1}_{i-1/2}),\nu^k(\boldsymbol{c}^{k,n-1}_{i+1/2}),U_{i-1}^{k,n},U_i^{k,n-1},U_{i+1}^{k,n-1})
    \\  
    & := 
   U^{k,n-1}_i- \lambda \big[
    \mathcal{F}^k(\nu^k(\boldsymbol{c}^{k,n-1}_{i+1/2}),U_i^{k,n-1},U_{i+1}^{k,n})
    - 
    \mathcal{F}^k(\nu^k(\boldsymbol{c}^{k,n-1}_{i-1/2}),U_{i-1}^{k,n-1},U_{i}^{k,n})
     \big]\\ 
     \label{scheme2}
     &:=U^{k,n-1}_i- \lambda \bigl[
    \mathcal{F}^{k,n}_{i+1/2} (U_i^{k,n-1},U_{i+1}^{k,n-1})
    - 
\mathcal{F}^{k,n-1}_{i-1/2} (U_{i-1}^{k,n-1},U_{i}^{k,n})
     \bigr].
     \end{split}\end{align}Here,  $\boldsymbol{c}_{i+1/2}^{k,n}:= \left(c_{i+1/2}^{s,k,n}\right)_{s\in \mathcal{N}}$ and 
$\mathcal{F}^k(\nu^k(\boldsymbol{c}^{k,n}_{i+1/2}),U_i^{k,n},U_{i+1}^{k,n})
    $ denotes the numerical approximation of the flux $f^k(U^k)\nu^k((\boldsymbol{\Theta}\circledast  \boldsymbol{U})^k)$ at $(t^n,x_{i+1/2})$ for $(k,i,n)\in \mathcal{N}\times\Z\times\mathcal{N}_T$, where for every $s,k\in \mathcal{N}$, 
\begin{align*}c_{i+1/2}^{s,k,n}&:=\Delta x \D t \sum\limits_{m=0}^n\sum\limits_{p\in\Z} \Theta^{s,k,n-m}_{i+1/2-p} U^{s,m}_{p}=\Delta x \D t \sum\limits_{m=0}^n\sum\limits_{p\in\Z} \Theta^{s,k,m}_{i+1/2-p} U^{s,n-m}_{p}\end{align*} approximates $\int_0^{t^n}\int_{\R}\Theta^{s,k}(x_{i+1/2}-\xi,t^n-\tau) U^{s,\D}(\tau,\xi )\d \xi \d \tau$. 

 Further, $\Theta_p^{j,k,s}=\mu^{j,k,p} \Gamma^{j,k,s},$ with $\mu^{j,k,p}$ and $\Gamma^{j,k,s}$ as the integral averages of $\mu^{j,k}$ and $\Gamma^{j,k}$ over $C_p$ and $C_s$ respectively.
In general, $\mathcal{F}^k$ can be defined as an appropriate nonlocal extension of any monotone numerical flux, meant for local conservation laws, for example,
\begin{equation*}
  \mathcal{F}^k(a_1,a_2,a_3)
   = 
  \frac{a_1}{2}\Big( f^k(a_2)
    +
    f^k(a_3)\Big)
  -
  \beta\frac{(a_3-a_2)}{2\, \lambda}, \beta\in (0,2/3),
\end{equation*} is an extension of Lax-Friedrich's flux. This flux will be used in the sequel,
where  $\Delta t$ is chosen in order to satisfy
the CFL condition
\begin{equation}\label{CFL_LF}
   \lambda \le \frac{\min(1, 4-6\beta,6\beta)}{1+6\abs{\boldsymbol{f}}_{(\lip(\R))^N}\norma{\boldsymbol{\nu}}_{(L^\infty(\R^N))^N}}.\end{equation}


\begin{remark}\normalfont
The choice of $\Gamma^{j,k,s}$ as the integral average will play a crucial role in proving the asymptotic compatibility in Section~\ref{NLL}. However, the convergence of the finite volume schemes,  i.e., the following stability estimates (cf.~Lemma~\ref{stability}) and the convergence  (cf.~Theorem~\ref{convergence})--hold for the pointwise evaluation approximation of $\boldsymbol{\Gamma}$ at the cell center as well.
\end{remark}

\begin{lemma}\label{stability}Let $(j,i,k,n,t)\in\mathcal{N}\times\Z \times \mathcal{N}\times \mathcal{N}_T\times\R^{+}$ and let the CFL condition~\eqref{CFL_LF} hold. 
The numerical approximations generated by the above marching formula satisfy:
\begin{enumerate}[label=(\alph*)]
 \item \label{lem:monotone} [Monotonicity] For a given fixed sequence
$\boldsymbol{c}_{i+1/2}^{n},$ $H^k$ is increasing in the last three arguments.
\item \label{Pos} [Invariant region principle]  
$0\leq U^{k,n}_i \leq 1.$
   \item \label{cons} [Conservation] 
$\sum\limits_{i\in \Z}U_i^{k,n} =  \sum\limits_{i\in \Z}U_i^{k,0}.$
\item \label{lem:L1} [$L^1$ bound]
    ${\norma{ {U^{k,\Delta}} (t)}_{L^1(\R)}} = {\norma{ {U^{k,\Delta}} (0)}_{L^1(\R)}}.$


\item \label{lem:BV}[ $\operatorname{BV}$ estimate] 

    \begin{align*}
    \sum\limits_{i\in\mathbb{Z}}
        \modulo{U_{i+1}^{k,n}-U_i^{k,n}} 
          \leq\left(
\exp(\mathcal{C}_7^kt)\sum\limits_{i\in\Z}\abs{U_{i+1}^{k,0}-U_i^{k,0}}  + \frac{\exp(\mathcal{C}_7^kt)-1}{\mathcal{C}_7^k}\mathcal{C}_8^k\right),
\end{align*}
where
\begin{align*}  
\mathcal{C}_7^k&=\mathcal{C}_5^k \abs{f^k}_{\lip(\R)}  { \norma{{\color{black}\nabla }{\nu^k}}_{(L^{\infty}(\R^N))^{N}}},\\
\mathcal{C}_8^k&=\mathcal{C}_6^k\abs{f^k}_{\lip(\R)}\norma{{\color{black}\nabla }{\nu^k}}_{(L^{\infty}(\R^N))^{N}}\norma{U_{0}^k}_{L^1(\R)}{+2(\mathcal{C}^k_5)^2\abs{f^k}_{\lip(\R)}\abs{{\color{black}\nabla }{\nu^k}}_{(\lip(\R^N))^{N}}}\norma{U_{0}^k}_{L^1(\R)}.
      \end{align*}
      \item \label{lem:L1t}[Time Estimate]
For $m>n\in \mathcal{N}_T,$ we have the following time estimate:
\begin{align*}
\Delta x\sum\limits_{i\in \mathbb{Z}} \abs{U_i^{k,m}-U_i^{k,n}} \leq \mathcal{C}_9^k \Delta t (m-n),
\end{align*}
where $\mathcal{C}_9^k$ depends on 
  $\sum\limits_{i\in\mathbb{Z}}
        \modulo{U_{i+1}^{k,n}-U_{i}^{k,n}} ,\norma{{\color{black}\nabla }{\nu^k}}_{(L^{\infty}(\R^N))^{N}},\norma{U_{0}^k}_{L^1(\R)},\abs{\mu^{j,k}}_{\operatorname{BV}(\R)},$\\ $  \norma{\Gamma^{j,k}}_{L^1(\R^+)}$ and is independent of $\delta$. 
        \item 
  \label{lem:entropy}[Discrete entropy condition]
For all $\alpha\in \R,$  we have
 \begin{align*} 
 \begin{split}
 &\modulo{U_i^{k,n+1}-\alpha}-\modulo{U_i^{k,n}- \alpha}+\lambda\left({{G^{k,n}_{{i+1/2}}}(U_i^{k,n} ,U_{i+1}^{k,n})}-G^{k,n}_{{i-1/2}}(U_{i-1}^{k,n} ,U_i^{k,n})\right)\\
 &\quad \quad +\lambda\sgn(U_i^{k,n+1}-\alpha)f^k(\alpha)(\nu^k(\boldsymbol{c}_{i+1/2}^{k,n})-\nu^k(\boldsymbol{c}_{i-1/2}^{k,n}))\le 0,
 \end{split}
 \end{align*}
 \begin{flalign*}&\text{where  }
{G^{k,n}_{{i+1/2}}}{(a,b)}=\mathcal{F}^{k,n}_{{i+1/2}}\left({\max(a,\alpha)},\,{\max(b,\alpha)}\right)
  -
  \mathcal{F}^{k,n}_{{i+1/2}}\left({\min(a,\alpha)},\,{\min(b,\alpha)}\right).&
\end{flalign*} 
 \item \label{eq:c_i+1/2} The convolution terms $\boldsymbol{c}$ satisfy the following bounds:
 \begin{align} \label{eq:c1}
&0 \leq {c}^{j,k,n}_{i+1/2} \leq 1,
\end{align}
\begin{align}
\label{eq:c2}
  &  | c^{j,k,n}_{i+1/2} - c^{j,k,n}_{i-1/2}|
    \leq 
    \mathcal{C}_5^k\Delta x,     \end{align}
     where $\mathcal{C}_5^k=\norma{U_{0}^k}_{L^1(\R)}\norma{\dot{\mu}^{j,k}}_{L^\infty(\R)}\norma{\Gamma^{j,k}}_{L^1(\R^+)},$  
     
    \begin{align}\label{eq:c3} &
    |c^{j,k,n}_{i+3/2} - 2c^{j,k,n}_{i+1/2}+c^{j,k,n}_{i-1/2}|
    \leq 
 \mathcal{C}_6^k\, \Delta x^{2},
 \end{align} 
 where $\mathcal{C}_6^k=2\norma{U_{0}^k}_{L^1(\R)} {\norma{\ddot{\mu}^{j,k}}_{L^{\infty}(\R)}} \norma{\Gamma^{j,k}}_{L^1(\R^+)}.$
 
      \end{enumerate}
    \end{lemma} 
    \begin{proof}We skip here; the proofs of \ref{lem:monotone}-\ref{lem:entropy} and \ref{eq:c_i+1/2}(\ref{eq:c1}) are quite classic by now in the literature of nonlocal conservation laws; see, for example, \cite{ACT2015, AHV2023, AV2023,ACG2015} and references therein. {\color{black}Items~\ref{lem:BV} and \eqref{eq:c2}-\eqref{eq:c3} are crucial and would yield spatial bounds that depend solely on $\|\Gamma^{j,k}\|_{L^1(\mathbb{R})}$ and are independent of $\|\Gamma^{j,k}\|_{L^{\infty}(\mathbb{R})}$. This property would enable the asymptotic passage from the memory-to-memoryless limit in \S\ref{NLL}. Therefore, in what follows, we prove these estimates one by one.} Note that \begin{align*}
\abs{c^{j,k,n}_{i+1/2} - c^{j,k,n}_{i-1/2}}
&= \Delta x \D t \bigg|\Delta_x^+\left(\sum\limits_{m=0}^n\sum\limits_{p\in\Z} \Theta^{j,k,n-m}_{i-1/2-p} U^{j,m}_{p}\right)\bigg|\\&= \Delta x\Delta t\sum\limits_{m=0}^n\sum\limits_{p\in\mathbb Z}
\bigg|\Delta_x^+\left(\mu^{j,k}_{i-1/2-p}\right)\bigg||\Gamma^{j,k}(t^{n-m})|U^{j,m}_p \\
&\le \Delta t\sum\limits_{m=0}^n|\Gamma^{j,k}(t^{n-m})|\Delta x\sum\limits_{p\in\mathbb Z}U^{j,m}_p\int_{C_{i-p}} |\dot{\mu}^{j,k}(s)|\,ds\le \mathcal{C}_5^k \D x,\end{align*} 
where
$\mathcal{C}_5^k:=\norma{U_{0}^k}_{L^1(\R)}\norma{\dot{\mu}^{j,k}}_{L^\infty(\R)}\norma{\Gamma^{j,k}}_{L^1(\R^+)},$
proving \eqref{eq:c2}. By a  repeated use of the mean value theorem on ${\mu}^{j,k}$ and ${\dot{\mu}}^{j,k},$ we have, 
\begin{align*}
    \begin{split}
    &\abs{( c^{j,k,n}_{i+3/2} - c^{j,k,n}_{i+1/2}) -( c^{j,k,n}_{i+1/2} - c^{j,k,n}_{i-1/2}) }
    \\&= 
    \Delta x \D t \bigg|\Delta_x^+\left(\sum\limits_{m=0}^n\sum\limits_{p\in\Z} \Theta^{j,k,n-m}_{i+1/2-p} U^{j,m}_{p}\right)-\Delta_x^+\left(\sum\limits_{m=0}^n\sum\limits_{p\in\Z} \Theta^{j,k,n-m}_{i-1/2-p} U^{j,m}_{p}\right) \bigg|\\
    &\le 
     \Delta x \D t \sum\limits_{k=0}^n\bigg|\sum\limits_{p\in\Z} U^{j,m}_{p}\Theta^{j,k,n-m}_{i-p-1} -2\sum\limits_{p\in\Z} U^{j,m}_{p}\Theta^{j,k,n-m}_{i-p}+\sum\limits_{p\in\Z} U^{j,m}_{p}\Theta^{j,k,n-m}_{i-p+1} \bigg|\\
      & \le 
    \Delta x \D t \sum\limits_{m=0}^n\sum\limits_{p\in\Z}U^{j,m}_{p}\bigg|\Delta_x^+\Theta^{j,k,n-m}_{i+1/2-p}-\Delta_x^+\Theta^{j,k,n-m}_{i-1/2-p}\bigg|= \mathcal{C}_6^k (\D x)^2 ,
    \end{split}
\end{align*}
where $C_6:=2\norma{U_{0}^k}_{L^1(\R)} {\norma{\ddot{\mu}^{j,k}}_{L^{\infty}(\R)}} \norma{\Gamma^{j,k}}_{L^1(\R^+)}.$  This completes the proof of \eqref{eq:c3}. To prove \ref{lem:BV}, we use the incremental form of \eqref{scheme2}, 
       the CFL condition \eqref{CFL_LF}, and repeating the Harten's argument (see \cite[Lem.~3.12]{HR2015}), we have 
   \begin{align}
\nonumber     U_{i+1}^{k,n+1}-U_{i}^{k,n+1}
&:=\mathcal{A}^{k,n}_{i,1}+\mathcal{A}^{k,n}_{i,2}+\mathcal{A}^{k,n}_{i,3},\end{align}
      where
        \begin{align*}
&\mathcal{A}^{k,n}_{i,1}=(U_{i+1}^{k,n}-U_{i}^{k,n}) \left(1-
      a^{k,n}_{i+1/2}-
      b^{k,n}_{i+1/2}\right)+
      a^{k,n}_{i-1/2}(U_{i}^{k,n}-U_{i-1}^{k,n}) +
      b^{k,n}_{i+1/2}(U_{i+2}^{k,n}-U_{i+1}^{k,n})
      \\
     & \mathcal{A}^{k,n}_{i,2}=-\lambda \left(
        {\color{black}\mathcal{F}^{k,n}_{{i+3/2}} (U_{i+1}^{k,n},U_{i+1}^{k,n})}
        -
          {\color{black}\mathcal{F}^{k,n}_{{i+1/2}}(U_{i+1}^{k,n},U_{i+1}^{k,n})}
\right),\\
&\mathcal{A}^{k,n}_{i,3}=\lambda\left({\color{black}\mathcal{F}^{k,n}_{{i+1/2}} (U_{i}^{k,n},U_{i}^{k,n})}
        -
          {\color{black}\mathcal{F}^{k,n}_{{i-1/2}} (U_{i}^{k,n},U_{i}^{k,n})}
      \right),\\
      & a^{k,n}_{i-1/2}
     = 
    \lambda \,
    \frac{ {\color{black}\mathcal{F}^{k,n}_{{i-1/2}} (U_{i}^{k,n},U_{i}^{k,n})}
      - \mathcal{F}^{k,n}_{{i-1/2}} (U_{i-1}^{k,n},U_{i}^{k,n})}{(U_{i}^{k,n}-U_{i-1}^{k,n})},
    \\& b^{k,n}_{i+1/2}
     = 
    \lambda \,
    \frac{
{\color{black}\mathcal{F}^{k,n}_{{i+1/2}} (U_{i}^{k,n},U_{i+1}^{k,n})}
      -
      \mathcal{F}^{k,n}_{{i+1/2}} (U_{i}^{k,n},U_{i}^{k,n})}{(U_{i+1}^{k,n}-U_{i}^{k,n})}.
   \end{align*}
With similar calculations as in  \cite{ACT2015} and using \eqref{CFL_LF}, we have $0\le  a^{k,n}_{i+1/2},  b^{k,n}_{i+1/2}\le 1/3,$ implying
$$\sum\limits_{i\in\Z}\modulo{\mathcal{A}^{k,n}_{i,1}}\le\sum\limits_{i\in\Z}\modulo{U_{i+1}^{k,n}-U_{i}^{k,n}}.$$
   {Further,}
\begin{align*}\begin{split}\frac{\mathcal{A}^{k,n}_{i,2}+\mathcal{A}^{k,n}_{i,3}}{\lambda}
       &= 
      - \left((f^k(U^{k,n}_{i+1})-f^k(U^{k,n}_{i})\right)
        \left( (\nu^k(\boldsymbol{c}^{k,n}_{{i+3/2}})-\nu^k(\boldsymbol{c}^{k,n}_{{i+1/2}})\right)\\
     &\quad -  f^k(U^{k,n}_{i})\left((\nu^k(\boldsymbol{c}^{k,n}_{{i+3/2}})-\nu^k(\boldsymbol{c}^{k,n}_{{i+1/2}}))-(
         \nu^k(\boldsymbol{c}^{k,n}_{{i+1/2}})-\nu^k(\boldsymbol{c}^{k,n}_{{i-1/2}}))
      \right).
      \end{split}
    \end{align*}
Using the facts that {$f^k(0)=0$ and $0\le U_i^{k,n}\le 1,$} we have,
   $$|f^k(U^{k,n}_{i})|
    \leq \abs{f^k}_{\lip(\R)} U_i^{k,n}\le \abs{f^k}_{\lip(\R)}.$$
 Finally, using the above estimate in addition to Lemma~\ref{eq:c_i+1/2} and {mean value theorem on $\nu^k$ and $\nabla{\nu}^k$}, we have
\begin{align*}
& |{\mathcal{A}^{k,n}_{i,2}+\mathcal{A}^{k,n}_{i,3}}|
  \\&\le\lambda \modulo{f^k(U^{k,n}_{i+1})-f^k(U^{k,n}_{i})}
  \norma{{\color{black}\nabla }{\nu^k}}_{(L^{\infty}(\R^N))^{N}}\modulo{\boldsymbol{c}^{k,n}_{i+3/2}-\boldsymbol{c}^{k,n}_{i+1/2}} \\&\quad+
      \lambda \abs{f^k}_{\lip(\R)}\norma{{\color{black}\nabla }{\nu^k}}_{(L^{\infty}(\R^N))^{N}}\abs{U_i^{k,n}}\modulo{\boldsymbol{c}^{k,n}_{i+3/2}-2\boldsymbol{c}^{k,n}_{i+1/2}+\boldsymbol{c}^{k,n}_{i+1/2}}\\
         &{\quad+  \lambda \abs{f^k}_{\lip(\R)}\abs{{{\color{black}\nabla }}{\nu^k}}_{(\lip(\R^N))^{N}}\modulo{\boldsymbol{c}^{k,n}_{i+1/2}-\boldsymbol{c}^{k,n}_{i-1/2}}\abs{U_i^{k,n}}\left(\modulo{\boldsymbol{c}^{k,n}_{i+1/2}-\boldsymbol{c}^{k,n}_{i-1/2}}+\modulo{\boldsymbol{c}^{k,n}_{i+3/2}-\boldsymbol{c}^{k,n}_{i+1/2}}\right)}
        \\
         &\le\Delta t\abs{{\color{black}\nabla }{\nu^k}}_{(\lip(\R^N))^{N}}\abs{f^k}_{\lip(\R)}\modulo{U_{i}^{k,n}-U_{i}^{k,n}}\mathcal{C}_5^k \\&\quad +
    \abs{f^k}_{\lip(\R)}\norma{{\color{black}\nabla }{\nu^k}}_{(L^{\infty}(\R^N))^{N}} \abs{U_i^{k,n}}\mathcal{C}_6^k\Delta x\Delta t\\
         &\quad+  2\abs{f^k}_{\lip(\R)}\abs{{{\color{black}\nabla }}{\nu^k}}_{(\lip(\R^N))^{N}}\abs{U_i^{k,n}}(\mathcal{C}^k_5)^2\Delta x\Delta t.  \end{align*}
Now, summing up  $|{\mathcal{A}^{k,n}_{i,2}+\mathcal{A}^{k,n}_{i,3}}|$ over all $i$,
        \begin{align}
   \nonumber \sum\limits_{i\in\Z} |{\mathcal{A}^{k,n}_{i,2}+\mathcal{A}^{k,n}_{i,3}}|&\le \mathcal{C}_5^k\Delta t  \abs{f^k}_{\lip(\R)}  { \norma{{\color{black}\nabla }{\nu^k}}_{(L^{\infty}(\R^N))^{N}}}\sum\limits_{i\in\Z}\abs{U_{i+1}^{k,n}-U_{i}^{k,n}}\nonumber\\&\nonumber\quad+ \mathcal{C}_6^k\Delta t \abs{f^k}_{\lip(\R)}\norma{{\color{black}\nabla }{\nu^k}}_{(L^{\infty}(\R^N))^{N}}\norma{U_{0}^k}_{L^1(\R)} \\ \nonumber
      &\quad+2(\mathcal{C}^k_5)^2\Delta t\abs{f^k}_{\lip(\R)}\abs{{{\color{black}\nabla }}{\nu^k}}_{(\lip(\R^N))^{N}}\norma{U_{0}^k}_{L^1(\R)}\\  \label{BC}  
       &=\mathcal{C}_7^k\Delta t \sum\limits_{i\in\Z}\abs{U_{i+1}^{k,n}-U_{i}^{k,n}} 
       + \mathcal{C}_8^k\Delta t,
       \end{align}
       finishing the proof.
     \end{proof}

{\begin{theorem}[Convergence]\label{convergence}
Assume that (\textbf{H1})--(\textbf{H3}) hold. As $\Delta x \rightarrow 0$, the approximations $\boldsymbol{U}^{\D}$ 
generated by the  marching formula \eqref{scheme2} converge in $(L^1_{\loc}(\overline{Q}_T))^N$ and pointwise a.e.~in $\overline{Q}_T$ to the entropy solution 
	$\boldsymbol{U}^{\D}\in (L^{\infty}({\overline{Q}_T}))^N \cap (C([0,T];L^1(\R{;[0,1]})))^N$ of the Cauchy problem \eqref{nlm}-\eqref{init}
  with initial data $\boldsymbol{U}_0 \in ((L^1 \cap \operatorname{BV})(\R{;[0,1]}))^N$.
\end{theorem}
\begin{proof}
Lemma \ref{lem:BV} implies that the sequence of functions $\boldsymbol{U}^{\D}(t,\cdot)$ is uniformly total variation bounded. Owing to the time estimate (cf.~Lemma \ref{lem:L1t}) and Helly's theorem (see \cite[Cor.~A.10]{HR2015}) there exists $\boldsymbol{U} \in (L^{\infty}([0,T]; \operatorname{BV}(\R)))^N \cap (C([0,T];L^1(\R{;[0,1]})))^N$ such that up to a subsequence ${\boldsymbol{U}^{\D}} \rightarrow \boldsymbol{U}$ in $(L^1_{\loc}(Q_T))^N.$ 
A Lax-Wendroff type argument presented in  \cite[Thm.~5.1]{BBKT2011} implies that the limit $U^k$ indeed satisfies the entropy condition def.~\eqref{kruz2}. 

\end{proof}

}
The above theorems imply that the entropy solution satisfies the following regularity estimates.
\begin{corollary}[Regularity of the entropy solution]\label{regu}Assume that (\textbf{H1})--(\textbf{H3}) hold.
For $0 <t \leq T$ and $U_{0}^k\in (L^1\cap\bv) (\R;[0,1]),$ the entropy solution $U^k$ of the IVP \eqref{nlm}--\eqref{init} satisfies the following:
\begin{align*}
\norma{U^k(t,\dott)}_{L^{\infty}(\R)}&\leq 1,\\
\norma{U^k(t,\dott)}_{L^1(\R)}&{\color{black}=}\norma{U_{0}^k}_{L^1(\R)},\\ 
\TV(U^k(t,\dott)) &\leq  \left(
     \exp(\mathcal{C}_7^kt)\sum\limits_{i\in\Z}\abs{\TV(U_{0}^k)}  + \frac{\exp(\mathcal{C}_7^kt)-1}{\mathcal{C}_7^k}\mathcal{C}_8^k\right),\\ 
\norma{U^k(t_2,\dott)-U^k(t_1,\dott)}_{L^1(\R)} &\leq \mathcal{C}_9^k\abs{t_2-t_1}, \text{ where } 0\leq t_1,t_2 \leq T.
\end{align*}
\end{corollary}
\begin{remark}\normalfont
  The above discussions in \S\ref{uni} and \S\ref{num} show that entropy solutions of the IVP generate a Lipschitz continuous solution operator $S_t$ on $(\bv(\R))^N$. It is to be noted that the solution operator is not a semigroup in the usual sense because $S_{t+s}(u_0) \neq S_t(S_s(u_0))$. Also, the $\bv$ bounds for $\boldsymbol{U}$ depend solely on $\|\Gamma^{j,k}\|_{L^1(\mathbb{R})}$ and are independent of $\|\Gamma^{j,k}\|_{L^{\infty}(\mathbb{R})}$. In the sequel we invoke this fact to study the memory-to-memoryless Dynamics
\end{remark}
\section{Memory-to-Memoryless Dynamics and Asymptotic Compatibility}\label{NLL}
 In this section, we study the limiting behavior of  the entropy solution of the ``space-time" nonlocal conservation law \eqref{nlm}-\eqref{init} as the kernel $\boldsymbol{\Gamma}$ converges to the Dirac delta distribution. We establish the following results:
\begin{enumerate}[(i)]
\item Theorem~\ref{con} shows that  the entropy solution of \eqref{nlm}--\eqref{init} converges strongly as $\delta \to 0^+$ to the entropy solution of the corresponding ``nonlocal-space'' conservation law \eqref{nls}, with initial data \eqref{init}.
\item Theorem~\ref{rate} shows that the finite volume approximation to \eqref{nlm}--\eqref{init} obtained in \S\ref{num} is asymptotically compatible with the above passage.
\end{enumerate}
This justifies the transition from the model with memory and its numerical approximation to its memoryless counterpart. Moreover, we establish explicit convergence rate estimates that quantify how fast this passage occurs as $\boldsymbol{\Gamma}$ concentrates at the origin. 

 We first introduce some notations and recall a lemma known as the relative entropy estimate for nonlocal conservation laws which is crucial to prove the error estimate. Let $j,k\in\mathcal{N}$ and $T>0$ be the final time. Let $\boldsymbol{\Gamma}$ satisfy \ref{H2A} such that $\Gamma^{j,k}$ has a support in a subset of $ [0,1],$ with $\displaystyle\int_{\R^+}\Gamma^{j,k}(\tau) \d \tau =1$. Further, for any $t>0,$ and for any $\delta>0$, we define,  
$\displaystyle\Gamma_{\delta}^{j,k}(t):=\frac{1}{\delta}\Gamma^{j,k}\left(\frac{t}{\delta}\right).$
We define $\Phi: \overline{Q}_T^2 \rightarrow \R$ by  
$\Phi(t,x,s,y):=\Phi^{\epsilon,\epsilon_0}(t,x,s,y)=\omega_{\epsilon}(x-y)\omega_{{\epsilon}_0}(t-s),$
where $\omega_a(x)=\displaystyle\frac{1}{a}\omega\left(\frac{x}{a}\right),$ $a>0$ and $\omega$ is a standard symmetric mollifier with $\operatorname{supp} (\omega) \subseteq [-1,1].$ In addition, we assume that $\displaystyle\int_\R \omega_a(x) \d x =1$ and $\displaystyle\int_\R\abs{\omega'_a(x)} \d x =\frac{1}{a}.$ Now, it is straightforward to see that $\Phi$ is symmetric and
$\Phi_x=\omega'_{{\epsilon}}(x-y)\omega_{\epsilon_0}(t-s)=-\Phi_y,  \Phi_t=\omega_{\epsilon}(x-y)\omega'_{{\epsilon}_0}(t-s)=-\Phi_s$.
For $a,b\in\R, \boldsymbol{Z}, \overline{\boldsymbol{Z}}\in (L^1(Q_T))^N,\phi \in C_c^{\infty}(\overline{Q}_T), \alpha \in \R$ and for $\epsilon,\epsilon_0>0$, let,
 \begin{align*}
 G^k(a,b)&:=\sgn (a-b) (f^k(b)-f^k(a)),\\
    \Lambda^k_T(Z^k,\phi,\alpha)&:= \int_{Q_T}\Big( |Z^k-\alpha|\phi_{t}+ \nu^k((\boldsymbol{Z}\circledast \boldsymbol{\mu})^k)G(Z^k,\alpha)\phi_{x}
     \\
     & \qquad  \quad  -  \sgn (Z^k-\alpha) f^k(\alpha) \partial_x(\nu^k((\boldsymbol{Z}\circledast \boldsymbol{\mu})^k)\phi\Big) \d t \d x\nonumber \\ & \quad  -\int_{\R}|Z^k(T,x)-\alpha|\phi(T,x)\d x +\int_{\R}|Z^k(0,x)-\alpha|\phi(0,x)\d x, \\
\Lambda^k_{\epsilon,\epsilon_0}(Z^k, \overline{Z}^k)&:=\int_{Q_T}\Lambda^k_T(Z^k,\Phi(\dott,\dott,s,y), \overline{Z}^k(s,y))\d y \d s,
    \\
 \gamma ({Z}^k,\sigma)&:=\sup_{\substack{
    \abs{t_1-t_2} \leq \sigma\\  0\leq t_1< t_2 \leq T }} \norma{{Z}^k(t_1,\dott)-{Z}^k(t_2,\dott)}_{L^1(\R)}.
  \end{align*}
An estimate on the errors $\norma{\boldsymbol{U}_{\delta}(T)-\boldsymbol{U}(T)}_{(L^1(\R))^N}$ and $\norma{\boldsymbol{U}_{\delta}^{\D}(T)-\boldsymbol{U}(T)}_{(L^1(\R))^N}$ would be achieved by estimating these differences in terms of $\Lambda^k$, which is the relative entropy functional wrt the target conservation law \eqref{nls}. To this end, we recall the relative entropy estimate for  “nonlocal-space” conservation law(see \cite[Lemma~3.3]{AHV2023}):
\begin{lemma}\label{lemma:kuz}[Relative entropy estimate]
Let $\boldsymbol{U}$ be the entropy solution  conservation law with memory \eqref{nlm}-\eqref{init} and
let $V$ belong to the set $\mathcal{K} := \Bigl\{ V : Q_T \to \mathbb{R}^N : 
\norma{V}_{(L^\infty(Q_T))^N} + \abs{V}_{(L^\infty_t {\bv}_x)^N} < \infty \Bigr\}.$ Then, the following estimate holds:
\begin{align}\label{kuz1}
\norma{\boldsymbol{U}(T,\dott)-\boldsymbol{V}(T,\dott)}_{(L^1(\R))^N}&\le\mathcal{C}_{10}\left(-\sum\limits_{k\in \mathcal{N}}\Lambda^k_{\epsilon,\epsilon_0}(V^k,U^k)+\sum\limits_{k\in \mathcal{N}}\gamma(V^k,{\epsilon_0})+\epsilon+\epsilon_0 \right), 
\end{align}
  where
  \begin{align*}
\mathcal{C}_{10}=\mathcal{C}_{10}(\boldsymbol{f},\boldsymbol{\mu}, \boldsymbol{\nu}, \norma{\boldsymbol{U}}_{(L^1(Q_T))N},\norma{\boldsymbol{V}}_{(L^1(Q_T))^N},|\boldsymbol{U}|_{(L^\infty_t \operatorname{BV}_x)^N},|\boldsymbol{V}|_{(L^\infty_t \operatorname{BV}_x)^N},|\boldsymbol{U}|_{(L^\infty_t \operatorname{BV}_x)^N},T)
  \end{align*}
and is independent of $\epsilon,\epsilon_0$.
\end{lemma}  Finally, we state and prove the main theorems of this section:
    \begin{theorem}\label{con}
    Let $\boldsymbol{U}$ be the entropy solution of the IVP \eqref{nls},\eqref{init} in the sense of \cite[Def 2.1]{ACG2015}, and let $\boldsymbol{U}_{\delta}$ be the entropy solution of the IVP \eqref{nlm}-\eqref{init} (cf.~Def.~\ref{def:entropy}) with $\Theta^{j,k}=\Theta^{j,k}_{\delta}:=\mu^{j,k}\Gamma^{j,k}_{\delta}$.
    Then, as $\delta \rightarrow 0,$ $\boldsymbol{U}_{\delta}$  converges to the entropy solution of \eqref{nls},\eqref{init} in $L^1$ norm and satisfies the following error estimate:
\begin{align*}
\norma{\boldsymbol{U}_{\delta}(T)-\boldsymbol{U}(T)}_{(L^1(\R))^N}=\mathcal{O}(\sqrt{\delta}).
\end{align*}
\end{theorem}


\begin{proof}

    {In view of the relative entropy estimate Lem.~\ref{lemma:kuz}, it is enough to estimate the relative entropy functional $-\Lambda^k_{\epsilon,\epsilon_0}(U^k_{\delta},U^k)$ and modulus of time continuity $\gamma(U_{\delta}^k,\epsilon_0)$.}
    We first prove that, for a fixed ${\delta}>0,$ and for any $\epsilon,\epsilon_0,$
    \begin{align*}
-\Lambda^k_{\epsilon,\epsilon_0}(U^k_{\delta},U^k) \leq \mathcal{C}_{11}^k \left({\delta}/{\epsilon} + {\delta} \right) \qquad \text{for all } k\in \mathcal{N},
\end{align*}
where $\mathcal{C}_{11}^k$ is given by \eqref{C}. Using the entropy inequality \eqref{kruz2} for $U^k_{\delta}$, we have,
\begin{align}\nonumber
&-\Lambda^k_{\epsilon,\epsilon_0}(U^k_{\delta},U^k) \\
 &\nonumber\le \int_{Q_T}\int_{Q_T}G^k(U^k_{\delta}(t,x),U^k(s,y))\nu^k((\boldsymbol{U}_{\delta}\circledast {\boldsymbol{\Theta}}_{\delta})^k)(t,x)) \Phi_x(s,y,t,x) \d t  \d x \d s \d y \\
 &\nonumber\quad - \int_{Q_T}\int_{Q_T}\sgn(U^k_{\delta}(t,x)-U^k(s,y))f^k(U^k(s,y))\partial_x(\nu^k((\boldsymbol{U}_{\delta}\circledast {\boldsymbol{\Theta}}_{\delta})^k)(t,x))) \\
 & \nonumber\qquad  \quad  \qquad \times 
 \Phi(s,y,t,x) \d t  \d x \d s \d y \\
 &\nonumber\quad-\int_{Q_T}\int_{Q_T}G^k(U^k_{\delta}(t,x),U^k(s,y))\nu^k((\boldsymbol{U}_{\delta}\circledast \boldsymbol{\mu})^k)(t,x)\Phi_x(s,y,t,x) \d t  \d x \d s \d y \\
 &\nonumber\quad + \int_{Q_T} \int_{\R}\sgn(U^k_{\delta}(t,x)-U^k(s,y))f^k(U^k(s,y))\partial_x(\nu^k((\boldsymbol{U}_{\delta}\circledast \boldsymbol{\mu})^k)(t,x))\Phi(s,y,t,x) \d t  \d x \d s \d y\\
 &\nonumber\le \int_{Q_T}\int_{Q_T} G^k(U^k_{\delta}(t,x),U^k(s,y))(\nu^k((\boldsymbol{U}_{\delta}\circledast {\boldsymbol{\Theta}}_{\delta})^k)(t,x))-\nu^k((\boldsymbol{U}_{\delta}\circledast \boldsymbol{\mu})^k)(t,x)\Phi_x(s,y,t,x) \d t  \d x \d s \d y \nonumber \\
 &\nonumber\quad +\int_{Q_T}\int_{Q_T}\sgn(U^k_{\delta}(t,x)-U^k(s,y))f^k(U^k(s,y)) 
 \nonumber \\&\nonumber
 \qquad \qquad\times(\partial_x(\nu^k((\boldsymbol{U}_{\delta}\circledast \boldsymbol{\mu})^k)(t,x))-\partial_x(\nu^k((\boldsymbol{U}_{\delta}\circledast {\boldsymbol{\Theta}}_{\delta})^k)(t,x)))) \Phi(s,y,t,x) \d t  \d x \d s \d y 
 \nonumber \\ &\nonumber \leq 
 \frac{1}{\epsilon}\abs{f^k}_{\lip(\R)} (\norma{U^k}_{L^{\infty}(Q_T)}+ \norma{U^k_{\delta}}_{L^{\infty}(Q_T)}) \abs{\nu^k}_{\lip(\R)} \int_{Q_T}\norma{(\boldsymbol{U}_{\delta}\circledast {\boldsymbol{\Theta}}_{\delta})^k(t,x)-(\boldsymbol{U}_{\delta}\circledast \boldsymbol{\mu})^k(t,x)} \d x \d t 
 \nonumber \\&\nonumber
 \quad + \abs{f^k}_{\lip(\R)} \norma{U^k}_{L^{\infty}(Q_T)} \int_{Q_T}\norma{\nabla\nu^k((\boldsymbol{U}_{\delta}\circledast {\boldsymbol{\Theta}}_{\delta})^k)(t,x)-\nabla\nu^k((\boldsymbol{U}_{\delta}\circledast \boldsymbol{\mu})^k)(t,x)} \norma{(\boldsymbol{U}_{\delta}\circledast \partial_x\boldsymbol{\mu})^k(t,x)} \d x \d t \nonumber\\
 &\nonumber
 \quad +  \abs{f^k}_{\lip(\R)} \norma{U^k}_{L^{\infty}(Q_T)} \int_{Q_T}  \norma{(\boldsymbol{U}_{\delta}\circledast \partial_x{\boldsymbol{\Theta}}_{\delta})^k(t,x)-(\boldsymbol{U}_{\delta}\circledast {\boldsymbol{\mu}}')^k(t,x)} \norma{\nabla\nu^k(\boldsymbol{U}_{\delta}\circledast {\boldsymbol{\Theta}}_{\delta})^k(t,x)} \d x \d t \nonumber
 \\ & :=\frac{1}{\epsilon}\abs{f^k}_{\lip(\R)} \abs{\nu^k}_{\lip(\R)}(\norma{U^k}_{L^{\infty}(Q_T)}+ \norma{U^k_{\delta}}_{L^{\infty}(Q_T)}) 
 I_1^k \nonumber \\&\quad +\abs{f^k}_{\lip(\R)} \norma{U^k}^2_{L^{\infty}(Q_T)} \norma{\boldsymbol{\mu}'}_{(L^1(\R))^{N^2}}
 \norma{\operatorname{Hess}\nu^k}_{(L^{\infty}(\R^N))^N} I_1^k\nonumber \\
  &\quad+\abs{f^k}_{\lip(\R)} \norma{U^k}_{L^{\infty}(Q_T)}  \norma{\nabla\nu^k}_{(L^{\infty}(\R^N))^N}I_2^k,\nonumber\\& \label{Gamma}
  \end{align}
   \vspace{-7mm}
 \begin{align*}
\text{where }\,I_1^k&:=\int_{Q_T}  \norma{(\boldsymbol{U}_{\delta}\circledast {\boldsymbol{\Theta}}_{\delta})^k(t,x)-(\boldsymbol{U}_{\delta}\circledast {\boldsymbol{\mu}})^k(t,x)} \d x \d t,\\
   I_2^k&:=\int_{Q_T}  \norma{(\boldsymbol{U}_{\delta}\circledast \partial_x{\boldsymbol{\Theta}}_{\delta})^k(t,x)-(\boldsymbol{U}_{\delta}\circledast {\boldsymbol{\mu}}')^k(t,x)} \d x \d t. 
\end{align*}
In what follows, we estimate the  terms $I_1^k$ and $I_2^k$.
Consider 
\begin{align*}
I_1^k&=\int_0^{\delta} \int_{\R}\norma{(\boldsymbol{U}_{\delta}\circledast {\boldsymbol{\Theta}}_{\delta})^k)(t,x)-(\boldsymbol{U}_{\delta}\circledast \boldsymbol{\mu})^k(t,x)} \d x \d t \\& \quad + \int_{\delta} ^T\int_{\R}\norma{(\boldsymbol{U}_{\delta}\circledast {\boldsymbol{\Theta}}_{\delta})^k(t,x)-(\boldsymbol{U}_{\delta}\circledast \boldsymbol{\mu})^k(t,x)} \d x \d t\\
&=\sum\limits_{j=1}^{N}\int_0^{\delta} \int_{\R}\Bigg| \int_0^t\int_{\R}U^j_{\delta}(\tau,\xi)\mu^{j,k}(x-\xi)\Gamma^{j,k}_{\delta}(t-\tau) \d \tau \d \xi  -  \int_{\R}U^j_{\delta}(t,\xi)\mu^{j,k}(x-\xi)  \d \xi \Bigg|  \d x \d t \\& \quad + \sum\limits_{j=1}^{N}\int_{\delta} ^T\int_{\R}\Bigg| \int_0^t\int_{\R}U^j_{\delta}(\tau,\xi)\mu^{j,k}(x-\xi)\Gamma^{j,k}_{\delta}(t-\tau) \d \tau \d \xi  -  \int_{\R}U^j_{\delta}(t,\xi)\mu^{j,k}(x-\xi)  \d \xi \Bigg|  \d x \d t
\end{align*}
Since $\displaystyle\int_0^{\delta} \Gamma^{j,k}_{\delta}(\tau)\d \tau =1$,   for every $j,k\in\mathcal{N},$ and for every 
$t\in [\delta,T], x\in \R$ we get,
\begin{align*}
\int_{\R}U^j_{\delta}(t,\xi)\mu^{j,k}(x-\xi)  \d \xi \d x &=\int_0^{t}\int_{\R}U^j_{\delta}(t,\xi)\mu^{j,k}(x-\xi) \Gamma^{j,k}_{\delta}(t-\tau)\d \tau \d \xi.
\end{align*}
Thus, 
\begin{align*}
I_1^k& 
\leq \sum\limits_{j\in\mathcal{N}}\int_0^{\delta} \int_{\R} \int_0^t\int_{\R}U^j_{\delta}(t-\tau,x-\xi)\mu^{j,k}(\xi) \Gamma^{j,k}_{\delta}(\tau)\d \tau \d \xi \d x\d t\\&\quad+\sum\limits_{j\in\mathcal{N}}\int_0^{\delta} \int_{\R} \int_{\R}U^j_{\delta}(t,x-\xi) \mu^{j,k}(\xi) \d \xi \d x\d t  \\
 & \quad + \sum\limits_{j\in\mathcal{N}}\int_{\delta}^T \int_{\R} 
 \int_0^{t}\int_{\R}
\abs{U^j_{\delta}(\tau,x-\xi)-U^j_{\delta}(t,x-\xi)}\mu^{j,k}(\xi) \Gamma^{j,k}_{\delta}(t-\tau)\d \tau \d \xi \d x \d t\\
& \leq 2\delta\norma{\boldsymbol{U}_{0}}_{(L^{1}(\R))^N} \norma{\boldsymbol{\mu}}_{(L^{1}(\R))^{N^2}}   +
{\delta}T\abs{\boldsymbol{U}_{\delta}}_{(\lip_t L^1_x)^N} \norma{\boldsymbol{\mu}}_{(L^{1}(\R))^{N^2}}:=\mathcal{C}_{12}{\delta}.
\end{align*}
\textit{Mutatis mutandis}, we have,
\begin{align*}
I_2^k&\leq 2\delta\norma{\boldsymbol{U}_{0}}_{(L^{1}(\R))^N} \norma{\boldsymbol{\mu'}}_{(L^{1}(\R))^{N^2}}  +
{\delta}T\abs{\boldsymbol{U}_{\delta}}_{(\lip_t L^1_x)^N} \norma{\boldsymbol{\mu}'}_{(L^{1}(\R))^{N^2}}:=\mathcal{C}_{13}{\delta}.
\end{align*}
Substituting the estimates on $I_1^k$ and $I_2^k$ in \eqref{Gamma}, we have,
\begin{align}\label{C}
\begin{split}
-\Lambda^k_{\epsilon,\epsilon_0}(U^k_{\delta},U^k)
&\leq \frac{{\delta}}{\epsilon}\abs{f^k}_{\lip(\R)} \abs{\nu^k}_{\lip(\R)}\left(\norma{U^k}_{L^{\infty}(Q_T)}+ \norma{U^k_{\delta}}_{L^{\infty}(Q_T)}\right) 
\mathcal{C}_{12}\\
&\qquad+{\delta}\abs{f^k}_{\lip(\R)} \norma{U^k}^2_{L^{\infty}(Q_T)} \norma{\boldsymbol{\mu}'}_{(L^1(\R))^{N^2}}
 \norma{\operatorname{Hess}\nu^k}_{(L^{\infty}(\R^N))^N}\mathcal{C}_{12}\\
  &\qquad+{\delta}\mathcal{C}_{13}\abs{f^k}_{\lip(\R)} \norma{U^k}_{L^{\infty}(Q_T)}  \norma{\nabla\nu^k}_{(L^{\infty}(\R^N))^N}:=\mathcal{C}_{11}^k\left(\displaystyle\frac{\delta}{\epsilon}+\delta\right),\end{split}
\end{align}
\vspace{-2mm}
 \begin{align*}
   \text{where} \, \, \mathcal{C}_{11}^k&:=\abs{f^k}_{\lip(\R)} \abs{\nu^k}_{\lip(\R)}\left(\norma{U^k}_{L^{\infty}(Q_T)}+ \norma{U^k_{\delta}}_{L^{\infty}(Q_T)}\right) 
\mathcal{C}_{12}\\
&\qquad+\abs{f^k}_{\lip(\R)} \norma{U^k}^2_{L^{\infty}(Q_T)} \norma{\boldsymbol{\mu}'}_{(L^1(\R))^{N^2}}
 \norma{\operatorname{Hess}\nu^k}_{(L^{\infty}(\R^N))^N}\mathcal{C}_{12}\\
  &\qquad+\abs{f^k}_{\lip(\R)} \norma{U^k}_{L^{\infty}(Q_T)}  \norma{\nabla\nu^k}_{(L^{\infty}(\R^N))^N}\mathcal{C}_{13}.
\end{align*}
   Now, since $\boldsymbol{U}_{\delta}$ is the entropy solution of \eqref{nlm}-\eqref{init}, we have
   \begin{align}\label{time}
   \gamma^k(U^k_{\delta},{\delta}) \leq \mathcal{C}_9^k{\delta},
   \end{align} owing to the time estimate (cf.~Cor.~\ref{regu}) with $\mathcal{C}_9^k$ independent of ${\delta}.$
   Finally, using \eqref{C}-\eqref{time} in Lemma~\ref{kuz1} and choosing $\epsilon=\epsilon_0=\sqrt{\delta},$ implies the theorem.
 \end{proof}
\begin{theorem}\label{rate}
 Let $\boldsymbol{U}$ be the entropy solution of the IVP \eqref{nls},\eqref{init} in the sense of \cite[Def 2.1]{ACG2015}, and let $\boldsymbol{U}^{\Delta}_{\delta}$ be the finite volume approximation to \eqref{nlm}-\eqref{init} obtained in \S\ref{num} with the space-time kernel $\boldsymbol{\Theta}_{\delta}$, where the mollifier $\Gamma$ satisfies \begin{equation}\label{Moment}
\displaystyle\int_{\mathbb{R}^{+}} x\,\Gamma^{j,k}(x)\,\d x\le C_{\boldsymbol{\Gamma}} < \infty.
 \end{equation} for $C_{\boldsymbol{\Gamma}}>0$. Then, as $\Delta x , \delta \rightarrow 0,$ $\boldsymbol{U}^{\Delta}_{\delta}$  converges to the entropy solution $\boldsymbol{U}$ of \eqref{nls},\eqref{init} in $L^1$ norm  and satisfies the following error estimate:
\begin{align*}
\norma{\boldsymbol{U}^{\Delta}_{\delta}(T)-\boldsymbol{U}(T)}_{(L^1(\R))^N} = \mathcal{O}(\sqrt{\delta}+\sqrt{\D x}).
\end{align*} 
\end{theorem}  
\begin{proof}
   {In view of the relative entropy estimate Lemma~\ref{lemma:kuz}, it is enough to estimate the relative entropy functional $-\Lambda^{k}_{\epsilon,\epsilon_0}(U^{k,\Delta}_{\delta},U^{k})$ and modulus of time continuity $\gamma(U_{\delta}^{k,\Delta},\epsilon_0)$. }
 We first prove that, for a fixed ${\delta}, \Delta x>0,$ and for any $\epsilon,\epsilon_0>0,$ the relative entropy functional satisfies the following estimate:
    \begin{align*}
-\Lambda^k_{\epsilon,\epsilon_0}(U_{\delta}^{k,\D},U^k) \leq \mathcal{C}_{14} \left({\delta}/{\epsilon} + {\delta} \right).
\end{align*}
Let $\sum\limits_{i,n}$ denote the double summation $\sum\limits_{i\in \Z}\sum\limits_{n=0}^{N_T-1}.$ Further, for $\alpha\in \R, \delta>0,i\in \Z, s,n,m\in {\mathcal{N}}_T, j,k \in\mathcal{N}, \alpha\in \R$ and $(t,x)\in Q_T,$ let 
$\Gamma^{j,k,s}_{\delta}=\displaystyle\frac{1}{\Delta t}\int_{C^s}\Gamma^{j,k}_{\delta}(\tau)d\tau,
    \eta_{i,\delta}^{k,n}(\alpha):=\abs{U_{i,\delta}^{k,n}-\alpha}$, $p_{i,\delta}^{k,n}(\alpha):=G(U_{i,\delta}^{k,n},\alpha) 
    $ and $c_{i+1/2}^{j,k,n,\delta}  :=\Delta x \Delta t\sum\limits_{m=0}^n\sum\limits_{p\in\Z} \Gamma^{j,k,n-m}_{\delta}\mu^{j,k}_{i+1/2-p} U^{j,m}_{p,\delta}$.
Using the fundamental theorem of calculus followed by summation by parts, $-\Lambda^k_{\epsilon,\epsilon_0}(U^{k,\D}_{\delta},U^k)$ can be written as:
\begin{align*}
-\Lambda^k_{\epsilon,\epsilon_0}(U_{\delta}^{k,\D},U^k) &=\lambda_1^k+\lambda_2^k+ \lambda_3^k,\end{align*}
\begin{align*}
\text{where }\lambda_1^k&:=\int_{Q_T} \sum\limits_{i,n} \left(\eta_{i,\delta}^{k,n+1}(U^k(s,y))-\eta_{i,\delta}^{k,n}(U^k(s,y))\right)\int\limits_{C_i}\Phi(s,y,t^{n+1},x)  \d x\d s  \d y, \nonumber\\
 \lambda_2^k&:=-\int_{Q_T} \sum\limits_{i,n} \int_{C_i^n}p_{i,\delta}^{k,n}(U^k(s,y))\nu^k((\boldsymbol{\mu}\circledast\boldsymbol{U}^{\Delta}_{\delta}(t))^k)(x)\Phi_x(s,y,t,x) \d t  \d x \d s \d y, \nonumber\\
 \lambda_3^k&:=\ \int_{Q_T}\sum\limits_{i,n}  \int_{C_i^n}\sgn(U_{i,\delta}^{k,n}-U^k(s,y))f^k(U^k(s,y))\mathcal{U}_x^{k,\Delta}(t,x)\Phi(s,y,t,x) \d t  \d x \d s \d y. \nonumber
\end{align*}
By applying the discrete entropy inequality \ref{lem:entropy} we get:
\begin{align*}
\lambda_1^k&\le
-\lambda \int_{Q_T} \sum\limits_{i,n}  \big(
\mathcal{G}^{k,n}_{i+1/2}(U_i^{k,n} ,U_{i+1}^{k,n},U^k(s,y) )-\mathcal{G}^{k,n}_{i-1/2}(U_{i-1}^{k,n} ,U_i^{k,n},U^k(s,y) )\big)\\&\qquad \qquad  \times \int_{C_i}\Phi(s,y,t^{n+1},x)  \d x\d s  \d y\\
 &\quad -\lambda \int_{Q_T} \sum\limits_{i,n}\sgn(U_i^{k,n+1}-U^k(s,y))f(U^k(s,y) )\left(\nu^k\left(\boldsymbol{c}_{i+1/2}^{k,n,\delta}\right) -\nu^k\left(\boldsymbol{c}_{i-1/2}^{k,n,\delta}\right)\right) \\&\qquad \qquad  \times 
\int_{C_i}\Phi(s,y,t^{n+1},x)  \d x\d s  \d y\\
&:= {\lambda}_{11}^k+{\lambda}_{12}^k.
\end{align*}
 Note that
\begin{align}
    \lambda_2^k&=-\int_{Q_T}\sum\limits_{i,n}  \int_{C_i^n}p_{i,\delta}^{k,n}(U^k(s,y))\nu^k\left(\boldsymbol{c}_{i+1/2}^{k,n,\delta}\right)\Phi_x(s,y,t,x)  \d t\d x\d s  \d y \nonumber\\ 
&\quad-\int_{Q_T}\sum\limits_{i,n}  \int_{C_i^n}p_{i,\delta}^{k,n}(U^k(s,y))\Phi_x(s,y,t,x) \left(\nu^k((\boldsymbol{\mu}\circledast\boldsymbol{U}^{\Delta}_{\delta}(t))^k)(x)-\nu^k\left(\boldsymbol{c}_{i+1/2}^{k,n,\delta}\right)\right)  \d t\d x\d s  \d y \nonumber\\
&:=\overline{\lambda}_2^{k}+\mathcal{E}_2^k, \nonumber\\&  \label{lam2}
 \end{align}
and
 \begin{align} \nonumber
    \lambda_3^k&= \int_{Q_T}\sum\limits_{i,n}  \sgn(U_i^{k,n}-U^k(s,y))f(U^k(s,y)) \left(\nu^k\left(\boldsymbol{c}_{i+1/2}^{k,n,\delta}\right)-\nu^k\left(\boldsymbol{c}_{i-1/2}^{k,n,\delta}\right)\right)\\ & \qquad \times \int\limits_{C^{n}} \Phi(s,y,t,x_{i+1/2})  \d t \d s \d y  \nonumber\\
    &\quad +\int_{Q_T}\sum\limits_{i,n}  \sgn(U_i^{k,n}-U^k(s,y))f(U^k(s,y))\nonumber \\
    & \quad \quad  \times 
    \int_{C_i^n}\mathcal{U}_x^{k,\Delta}(t,x)  \left(\Phi(s,y,t,x)-\Phi(s,y,t,x_{i+1/2})\right) \d x \d t  \d y \d s \nonumber\\
    &\quad + \int_{Q_T}\sum\limits_{i,n}  \sgn(U_i^{k,n}-U^k(s,y))f(U^k(s,y))\nonumber \\ 
    & \quad \quad \times 
\int_{C_i^n}\left(\mathcal{U}_x^{k,\Delta}(t,x)  -\frac1{\Delta x}\left(\nu^k\left(\boldsymbol{c}_{i+1/2}^{k,n,\delta}\right)-\nu^k\left(\boldsymbol{c}_{i-1/2}^{k,n,\delta}\right)\right)\right)\Phi(s,y,t,x_{i+1/2}) \d x \d t \d y \d s\nonumber\\ \nonumber
    &\quad :=\overline{\lambda}_3^{k}+\mathcal{E}_{31}^k+\mathcal{E}_{32}^k.\\& \label{lam3} 
\end{align}
To complete the proof of the theorem, it remains to prove the following claims:
\begin{enumerate}[\textbf{Claim \arabic*}]
        \item \label{Cl1}: $\lambda_{11}^{k}+\overline{\lambda}_2^{k}=\mathcal{O}\left(\displaystyle\frac{\Delta x}{\epsilon}+\displaystyle\frac{\Delta t}{\epsilon_0} \right)$
    \item \label{Cl2}:$\lambda_{12}^{k}+\overline{\lambda}_3^{k}=\mathcal{O}\left(\displaystyle\frac{\Delta x}{\epsilon}+\displaystyle\frac{\Delta t}{\epsilon_0} +\D t + \delta\right)$
     \item \label{Cl3}: $\mathcal{E}_2^k=\mathcal{O}\left(\displaystyle\frac{\Delta x}{\epsilon}+\displaystyle\displaystyle\frac{\delta}{\epsilon}\right)$
    \item \label{Cl4}: $\mathcal{E}_{31}^k+ \mathcal{E}_{32}^k=\mathcal{O}\left(\displaystyle\frac{\Delta x}{\epsilon}+\displaystyle\displaystyle\frac{\delta}{\epsilon}\right)$
\end{enumerate}
 Note that the aim is to get the estimates, in terms of both $\Delta x, \Delta t$ and $\delta$. Also, since \ref{Cl1} depends only on $\Delta x$ and $\Delta t$, and not on $\delta,$ it can be essentially done on the lines of \cite[Lemma 3.3]{AHV2023}, since \begin{align}\label{cdiff}
   \sum\limits_i| c^{j,k,n}_{i+1/2} - c^{j,k,n}_{i-1/2}|\le \norma{U_{0}^k}_{L^1(\R)}\abs{{\mu}^{j,k}}_{BV(\R)}\norma{\Gamma^{j,k}}_{L^1(\R^+)},
\end{align} by summing \eqref{eq:c2} over $i\in\Z$. Note that \ref{Cl2} can be essentially done on the lines of Claim 2 of \cite[Lemma 3.3]{AHV2023} provided that $\sum\limits_i \Big|
\bigl( c^{j,k,n}_{i+1/2} - c^{j,k,n}_{i-1/2} \bigr)
- \bigl( c^{j,k,n-1}_{i+1/2} - c^{j,k,n-1}_{i-1/2} \bigr)
\Big|$ does not depend on $\displaystyle\frac{1}{\delta}$. Infact, in what follows we prove that the following holds: \begin{align}\label{cd2}
&\sum\limits_i \Big|
\bigl( c^{j,k,n}_{i+1/2} - c^{j,k,n}_{i-1/2} \bigr)
- \bigl( c^{j,k,n-1}_{i+1/2} - c^{j,k,n-1}_{i-1/2} \bigr)
\Big| \le 
\mathcal{C}_{14}^k\, (\Delta t+\delta),
\end{align} which is a non trivial and requires the additional assumption (cf.~\ref{Moment}) on first moment of $\Gamma^{j,k}$ . 
Note that
\[
\sum\limits_{m\in\mathcal{N}_T} m\Delta t\,\Gamma_{\delta}^{j,k,m}
= \sum\limits_{m\in\mathcal{N}_T}m\int_{C^m}\Gamma^{j,k}_{\delta}(\tau)\,d\tau.
\]
For every $\tau\in C^m,$ we have $t_m\le \tau$.
Using the nonnegativity of $\Gamma^{j,k}_{\delta}$, this implies
\[
m\Delta t\,\Gamma^{j,k}_{\delta}(\tau)
\le \tau\,\Gamma^{j,k}_{\delta}(\tau).
\]
Therefore,
\begin{align*}
\sum\limits_{m\in\mathcal{N}_T} m\Delta t\,\Gamma_{\delta}^{j,k,m}
&\le \frac{1}{\Delta t}
\sum\limits_{m\in\mathcal{N}_T}\int_{C^m}
\tau\,\Gamma^{j,k}_{\delta}(\tau)\,d\tau= \frac{1}{\Delta t}\int_{0}^{\infty}
\tau\,\Gamma^{j,k}_{\delta}(\tau)\,d\tau=\frac{\delta}{\Delta t}\int_{0}^{\infty}\theta\,\Gamma^{j,k}(\theta)\,d\theta.
\end{align*}
Using \eqref{Moment}, we have,
\begin{align}\label{mo}
    \sum\limits_{m\ge 1} m\Delta t\,\Gamma_{\delta}^{j,k,m}
 \le \delta \frac{C_{\boldsymbol{\Gamma}}}{\Delta t},
\sum\limits_{m\ge 1} \Delta t\,\Gamma_{\delta}^{j,k,m}
 \le \delta \frac{C_{\boldsymbol{\Gamma}}}{\Delta t}.
\end{align}
Observe that,
\[
c^{j,k,n}_{i+\frac12}
=
\Delta x\,\Delta t
\sum\limits_{m=0}^n
\sum\limits_{p\in\mathbb Z}
\Gamma^{j,k,n-m}_\delta\,
\mu^{j,k}_{i+\frac12-p}\,
U^{k,m}_p =
\Delta x\,\Delta t
\sum\limits_{m=0}^{n}
\sum\limits_{p\in\mathbb Z}
\Gamma^{j,k,m}_\delta\,
\mu^{j,k}_{i+\frac12-p}\,
U^{k,n-m}_p,
\]
which implies that
\begin{align}
c^{j,k,n}_{i+\frac12}-c^{j,k,n}_{i-\frac12}
&=
\Delta x\,\Delta t
\sum\limits_{m=0}^n
\sum\limits_{p\in\mathbb Z}
\Gamma^{j,k,m}_\delta
\bigl(
\mu^{j,k}_{i+\frac12-p}-\mu^{j,k}_{i-\frac12-p}
\bigr)
U^{k,n-m}_p,
\label{eq:jump-n}
\end{align}
and 
\begin{align}
c^{j,k,n-1}_{i+\frac12}-c^{j,k,n-1}_{i-\frac12}
&=
\Delta x\,\Delta t
\sum\limits_{m=0}^{n-1}
\sum\limits_{p\in\mathbb Z}
\Gamma^{j,k,m}_\delta
\bigl(
\mu^{j,k}_{i+\frac12-p}-\mu^{j,k}_{i-\frac12-p}
\bigr)
U^{k,n-1-m}_p .
\label{eq:jump-n-1}
\end{align}
Subtracting \eqref{eq:jump-n-1} from \eqref{eq:jump-n} yields
\begin{align}
&\bigl( c^{j,k,n}_{i+\frac12}-c^{j,k,n}_{i-\frac12} \bigr)
-
\bigl( c^{j,k,n-1}_{i+\frac12}-c^{j,k,n-1}_{i-\frac12} \bigr)
\\&=
\Delta x\,\Delta t
\Bigg[
\sum\limits_{m=0}^{n-1}\Gamma^{j,k,m}_\delta
\sum\limits_{p}
\bigl(
\mu^{j,k}_{i+\frac12-p}-\mu^{j,k}_{i-\frac12-p}
\bigr)(U^{k,n-m}_p-U^{k,n-1-m}_p)
 \nonumber\\
&\quad 
+
\sum\limits_{p}
\Gamma^{j,k,n}_\delta
\bigl(
\mu^{j,k}_{i+\frac12-p}-\mu^{j,k}_{i-\frac12-p}
\bigr)
U^{k,0}_p
\Bigg].
\label{eq:raw-difference}
\end{align}
Then, using \ref{lem:L1t}, we have \eqref{cd2}, 
with $\mathcal{C}_{14}^k:=
\abs{\mu^{j,k}}_{\bv(\R)}\mathcal{C}_9^k
+
\,\displaystyle\frac{C_{\boldsymbol{\Gamma}}}{\lambda}
\abs{\mu^{j,k}}_{\bv(\R)}
\norma{{\boldsymbol{U}}^{0}}_{(L^1(\R))^N},$ which completes the proof of \ref{Cl2}. To prove \ref{Cl3},
consider the term,
\begin{align*} & \nonumber \sum\limits_{i,n} \int_{C_i^n} \left|\nu^k((\boldsymbol{\mu}\circledast\boldsymbol{U}^{\Delta}_{\delta}(t))^k)(x)-\nu^k\left(\boldsymbol{c}_{i+1/2}^{k,n,\delta}\right)\right|\d t\d x\\
&=\sum\limits_{i}\sum\limits_{n=0}^{n=n_{\delta} } \int_{C_i^n}\left|\nu^k((\boldsymbol{\mu}\circledast\boldsymbol{U}^{\Delta}_{\delta}(t))^k)(x)-\nu^k\left(\boldsymbol{c}_{i+1/2}^{k,n,\delta}\right)\right|\d t\d x\\
&\qquad+\sum\limits_{i}\sum\limits_{n=n_{\delta}+1}^{n=N_T} \int_{C_i^n}\left|\nu^k((\boldsymbol{\mu}\circledast\boldsymbol{U}^{\Delta}_{\delta}(t))^k)(x)-\nu^k\left(\boldsymbol{c}_{i+1/2}^{k,n,\delta}\right)\right|\d t\d x\\
&=\mathcal{E}_{21}^k+\mathcal{E}_{22}^k,
\end{align*}
 {where $n_{\delta}\in \mathcal{N}_T$ is taken to be the smallest integer such that $n_{\delta}\Delta t\geq\delta.$}
Note that,
\begin{align}
 \mathcal{E}_{21}^k
&\le\abs{\nu^k}_{\lip(\R)}\sum\limits_{i}\sum\limits_{n=0}^{n=n_{\delta} }\sum\limits_{p\in\Z}\sum\limits_{j\in\mathcal{N}}\int_{C_i^n}
 \bigg| \int_{C_p}\mu^{j,k}(x-\xi)  U^{j,n}_{p,\delta}\d \xi -\Delta x \D t \sum\limits_{m=0}^n \Theta^{j,k,n-m}_{i+1/2-p,\delta} U^{j,m}_{p,\delta}\bigg| \d t\d x  \nonumber \\
&\le \abs{\nu^k}_{\lip(\R)}\sum\limits_{i}\sum\limits_{n=0}^{n=n_{\delta} }\sum\limits_{p\in\Z}\sum\limits_{j\in\mathcal{N}}\int_{C_i^n}
 \int_{C_p}\mu^{j,k}(x-\xi)  U^{j,n}_{p,\delta}\d \xi \d t\d x \nonumber \\
&\qquad+\Delta x \D t \abs{\nu^k}_{\lip(\R)}\sum\limits_{i}\sum\limits_{n=0}^{n=n_{\delta} }\sum\limits_{p\in\Z}\sum\limits_{j\in\mathcal{N}}\sum\limits_{m=0}^n\int_{C_i^n} \Theta^{j,k,n-m}_{i+1/2-p,\delta} U^{j,m}_{p,\delta} \d t\d x  \nonumber
\\
&\le 2\abs{\nu^k}_{\lip(\R)}(\delta+\Delta t)\norma{\boldsymbol{U_0}}_{(L^1(\R))^{N}}\norma{\boldsymbol{\mu}}_{(L^1(\R))^{N^2}} +\abs{\nu^k}_{\lip(\R)}(\delta+\Delta t)\norma{\boldsymbol{U_0}}_{(L^1(\R))^{N}}\norma{\boldsymbol{\mu}}_{(L^1(\R))^{N^2}}\nonumber\\
&:=\mathcal{C}_{21}^k(\delta+\Delta t).\nonumber\\&\label{S1}\end{align}
To estimate $\mathcal{E}_{22}^k,$ the idea is to invoke the fact that, for any $t\ge \delta,$ $\displaystyle\int_0^t\Gamma^{j,k}_{\delta}(\tau)d\tau=1,$ which implies $\Delta t\displaystyle\sum\limits_{m=0}^n\Gamma^{j,k,m}_{\delta}=1$ for any $n\ge n_{\delta}.$
Thus,  \begin{align*} \allowdisplaybreaks
\mathcal{E}_{22}^k&= \sum\limits_{i}\sum\limits_{n=n_{\delta}+1}^{n=N_T} \sum\limits_{j\in\mathcal{N}}\int_{C_i^n}
 \bigg|\nu^k \bigg(\sum\limits_{p\in\Z} \int_{C_p}\mu^{j,k}(x-\xi)  U^{j,n}_{p,\delta}\d \xi \bigg) -\nu^k \bigg(\Delta x \D t \sum\limits_{m=0}^n\sum\limits_{p\in\Z} \Theta^{j,k,n-m}_{i+1/2-p,\delta} U^{j,m}_{p,\delta}\bigg)\bigg| \d t\d x \\
&\le \abs{\nu^k}_{\lip(\R)}\sum\limits_{i}\sum\limits_{n=n_{\delta}+1}^{n=N_T} \sum\limits_{j\in\mathcal{N}}\int_{C_i^n}\bigg|\sum\limits_{p\in\Z}\int_{C_p}\mu^{j,k}(x-\xi)  U^{j,n}_{p,\delta}\d \xi -\Delta x \D t \sum\limits_{m=0}^n \Theta^{j,k,n-m}_{i+1/2-p,\delta} U^{j,m}_{p,\delta}\bigg|  \\&
 \le \abs{\nu^k}_{\lip(\R)}
\sum\limits_{i}\sum\limits_{n=n_{\delta}+1}^{n=N_T} \sum\limits_{j\in\mathcal{N}}\\&
\qquad\times\int_{C_i^n}\bigg|\Delta t\sum\limits_{m=0}^n\Gamma^{j,k,m}_{\delta}\sum\limits_{p\in\Z}\int_{C_p}\mu^{j,k}(x-\xi) U^{j,n}_{p,\delta}\d \xi -\Delta x \D t \sum\limits_{m=0}^n \mu^{j,k}_{i+1/2-p}\Gamma^{j,k,m}_{\delta} U^{j,n-m}_{p,\delta}\bigg| \d t\d x \\&
\le \abs{\nu^k}_{\lip(\R)}\Delta t
\sum\limits_{i}\sum\limits_{n=n_{\delta}+1}^{n=N_T} \sum\limits_{m=0}^n \\& \qquad \times \Gamma^{j,k,m}_{\delta}\sum\limits_{j\in\mathcal{N}}\sum\limits_{p\in\Z}\int_{C_i^n}\bigg|\int_{C_p}\mu^{j,k}(x-\xi)   U^{j,n}_{p,\delta}\d \xi -\Delta x  \mu^{j,k}_{i+1/2-p} U^{j,n-m}_{p,\delta}\bigg| \d t\d x, 
\end{align*}
which implies,
\begin{align*}
\mathcal{E}_{22}^k&\le \abs{\nu^k}_{\lip(\R)}\Delta t\sum\limits_{i}\sum\limits_{n=n_{\delta}+1}^{n=N_T} \sum\limits_{m=0}^n\Gamma^{j,k,m}_{\delta}\sum\limits_{j\in\mathcal{N}}\sum\limits_{p\in\Z}\int_{C_i^n}
\bigg|\int_{C_p}\mu^{j,k}(x-\xi)   U^{j,n}_{p,\delta}\d \xi-\Delta x \mu^{j,k}_{i+1/2-p} U^{j,n}_{p,\delta}\bigg|\\
&\qquad+\abs{\nu^k}_{\lip(\R)}\bigg|\Delta x \mu^{j,k}_{i+1/2-p} U^{j,n}_{p,\delta}-\Delta x \mu^{j,k}_{i+1/2-p} U^{j,n-m}_{p,\delta}\bigg|\\
&\le \abs{\nu^k}_{\lip(\R)}\Delta t\sum\limits_{i}\sum\limits_{n=n_{\delta}+1}^{n=N_T} \sum\limits_{m=0}^n\Gamma^{j,k,m}_{\delta}\sum\limits_{j\in\mathcal{N}}\sum\limits_{p\in\Z}\int_{C_i^n}
\bigg|\int_{C_p}\mu^{j,k}(x-\xi)   U^{j,n}_{p,\delta}\d \xi-\int_{C_p} \mu^{j,k}_{i+1/2-p} U^{j,n}_{p,\delta}\d \xi\bigg|\\\
&\qquad+\abs{\nu^k}_{\lip(\R)}\bigg|\Delta x \mu^{j,k}_{i+1/2-p} U^{j,n}_{p,\delta}-\Delta x \mu^{j,k}_{i+1/2-p} U^{j,n-m}_{p,\delta}\bigg| \\
&\le \abs{\nu^k}_{\lip(\R)}\Delta t\sum\limits_{i}\sum\limits_{n=n_{\delta}+1}^{n=N_T} \sum\limits_{m=0}^n\Gamma^{j,k,m}_{\delta}\sum\limits_{j\in\mathcal{N}}\sum\limits_{p\in\Z}\int_{C_i^n}\int_{C_p}\abs{\mu^{j,k}(x_{i+1/2-p})-\mu^{j,k}(x-\xi)}U^{j,n}_{p,\delta}\d \xi \d\tau\d t\d x\\
&\qquad+\abs{\nu^k}_{\lip(\R)}\Delta t\sum\limits_{i}\sum\limits_{n=n_{\delta}+1}^{n=N_T} \sum\limits_{m=0}^n\Gamma^{j,k,m}_{\delta}\sum\limits_{j\in\mathcal{N}}\Delta x\sum\limits_{p\in\Z}\abs{U^{j,n}_{p,\delta}-U^{j,n-m}_{p,\delta}}\int_{C_i^n}\mu^{j,k}(x_{i+1/2-p}) \d t\d x
\\&=  \abs{\nu^k}_{\lip(\R)}\mathcal{E}_{221}^k+\abs{\nu^k}_{\lip(\R)}\mathcal{E}_{222}^k.
\end{align*}
Now, a direct estimate yields 
\begin{align*}
    \mathcal{E}_{221}^k\le N \Delta x\abs{\mu^{j,k}}_{BV(\R)}  \norma{\boldsymbol{U}_0}_{L^1(\R^N)}T:=\mathcal{C}_{221}^k\Delta x.
    \end{align*}
It should be noted that a direct estimate on $\mathcal{E}_{222}^k$ will yield a term $1/\delta$, which does not serve the purpose. Instead, 
in what follows, we invoke the moment estimates \eqref{Moment} to estimate $\mathcal{E}_{222}^k$. Firstly, we note that
using \eqref{mo} and \ref{lem:L1t} in $\mathcal{E}_{222}^k$, we have
\begin{align}
\begin{split} \label{e_222}
\mathcal{E}_{222}^k 
    &=\Delta t\Delta t\sum\limits_{n=n_{\delta}+1}^{n=N_T} \sum\limits_{m=0}^n\Gamma^{j,k,m}_{\delta}\sum\limits_{j\in\mathcal{N}}\Delta x\sum\limits_{p\in\Z}\abs{U^{j,n}_{p,\delta}-U^{j,n-m}_{p,\delta}}\norma{\mu^{j,k}}_{L^1(\R)}\\
    &=\Delta t^3 \norma{\mu^{j,k}}_{L^1(\R)} \sum\limits_{n=n_{\delta}+1}^{n=N_T} \sum\limits_{m=0}^n\Gamma^{j,k,m}_{\delta}\sum\limits_{j\in\mathcal{N}} m \mathcal{C}_9^k
        =\Delta t^3  N  \mathcal{C}_9^k\norma{\mu^{j,k}}_{L^1(\R)} \sum\limits_{n=n_{\delta}+1}^{n=N_T} \sum\limits_{m=0}^n m \Gamma^{j,k,m}_{\delta}
\\
    &=  \Delta t  N  \mathcal{C}_9^k\norma{\mu^{j,k}}_{L^1(\R)} \sum\limits_{n=n_{\delta}+1}^{n=N_T} C_{\boldsymbol{\Gamma}} \delta \nonumber =N\mathcal{C}_9^k\norma{\mu^{j,k}}_{L^1(\R)} TC_{\boldsymbol{\Gamma}}  \delta:=\mathcal{C}_{222}^k\delta
    \end{split}
\end{align}
Now, combining the above estimates, we get: 
\begin{align*}
&fds\nonumber\sum\limits_{i,n} \int_{C_i^n} \left|\nu^k((\boldsymbol{\mu}\circledast\boldsymbol{U}^{\Delta}_{\delta}(t))^k)(x)-\nu^k\left(\boldsymbol{c}_{i+1/2}^{k,n,\delta}\right)\right|\d t\d x \\ &\le \mathcal{C}_{21}^k(\delta+\Delta x)+\abs{\nu^k}_{\lip(\R)}\mathcal{C}_{221}^k\Delta x+\abs{\nu^k}_{\lip(\R)}\mathcal{C}_{222}^k\delta.
\end{align*}
Finally, using the following bounds for every $(t,x) \in Q_T$, \begin{align*}\int_{Q_T}\abs{p_{i,\delta}^{k,n}(U^k(s,y))} \d y \d s &\le \abs{f^k}_{\lip(\R)}
\norma{U^{k}_0}_{L^1(\R)},
\int_{Q_T}\abs{\Phi_x(s,y,t,x)} \d y \d s \le \displaystyle\frac{1}{\epsilon}, 
\end{align*}
we obtain that
\begin{align*}
\nonumber\abs{\mathcal{E}_2^k}&\le \abs{f^k}_{\lip(\R)}
\norma{U^{k}_0}_{L^1(\R)}\left(\mathcal{C}_{21}^k\displaystyle\frac{\delta}{\epsilon}+\lambda\mathcal{C}_{21}^k\displaystyle\frac{\Delta x}{\epsilon}+\abs{\nu^k}_{\lip(\R)}\mathcal{C}_{221}^k\frac{\Delta x}{\epsilon}+\abs{\nu^k}_{\lip(\R)}\mathcal{C}_{222}^k\displaystyle\frac{\delta}{\epsilon}\right),
\end{align*}  
which completes the proof of \ref{Cl3}. 
Further, note that
$\mathcal{E}^k_{31}$ and  $\mathcal{E}^k_{32}$ do not depend on $\norma{\Gamma_{\delta}^{j,k}}_{L^{\infty}(\R)}$ and only depends on $\norma{\Gamma_{\delta}^{j,k}}_{L^{1}(\R)},$ and hence, the proof of \ref{Cl4} can be handled on the same lines of \ref{Cl3}, (see also $\mathcal{E}_{31}$ and $\mathcal{E}_{32}$ in \cite[Lemma 3.3]{AHV2023}). To complete the proof, firstly we set 
$\varepsilon =\max(\sqrt{\delta},\sqrt{\Delta x})$ and $  \varepsilon_0 = \sqrt{\Delta t}.$ Now, note that  \ref{lem:L1t} implies $
\gamma(U^{k,\Delta}_{\delta}, \sqrt{\Delta t}) = \mathcal{O}(\sqrt{\Delta t})$. 
and \eqref{init} implies 
$
\|\boldsymbol{U}_0^\Delta - \boldsymbol{U}_0\|_{(L^1(\mathbb{R}))^{N}} = \mathcal{O}(\Delta x).$ Finally, the desired estimate follows from combining the results from \ref{Cl1}-\ref{Cl4} with Lemma \ref{kuz1} and 
the choice of $\epsilon$ and $\epsilon_0$.
\end{proof}
\section{Numerical Experiments}
\label{num1}
We present numerical experiments to illustrate the theory developed in \S\ref{NLL}, highlighting the memory-to-memoryless dynamics and the asymptotic compatibility of the finite volume scheme \eqref{scheme2}. Throughout this section, we choose $\beta = 0.3333$ and $\lambda = 0.1286$ so as to satisfy the CFL condition \eqref{CFL_LF}.
We consider a {nonlocal-in-space and nonlocal-in-time} generalization of the Keyfitz–Kranzer system introduced in \cite{ACG2015} in one spatial dimension, given by 
\begin{equation}
  \label{eq:kk}
  \left\{
    \begin{array}{l}
      \partial_t U^1 + \partial_x \left(U^1 \,  \nu (\Theta_{\delta}*U^1,\Theta_{\delta}*U^2)\right) = 0
      \\
      \partial_t U^2 + \partial_x \left( U^2 \,  \nu (\Theta_{\delta}*U^1,\Theta_{\delta}*U^2)\right) = 0
    \end{array}
  \right.,
\end{equation} where
$ \Gamma_{\delta}(t)=\displaystyle\frac{3}{\delta^3}(\delta-t)^2\mathbbm{1}_{(0,\delta)}(t), \mu(x)=\displaystyle Lx(\eta-x)^3\mathbbm{1}_{(0,\eta)}(x),\Theta_{\delta}(t,x)=\mu(x)\Gamma_{\delta}(t)$, and $\nu(a,b)=(1-a^2-b^2)^3$ where $L$ is such that $\displaystyle\int_{\R}\mu(x)\d x=1.$
 The system~\eqref{eq:kk} fits into the framework of this article  with $N=2, \boldsymbol{\Theta}_{\delta} =
    \left[\begin{array}{ccc}
        \Theta_{\delta} & \Theta_{\delta}
        \\
\Theta_{\delta}& \Theta_{\delta}
      \end{array}\right],\nu^k(x)=\nu(x),$ and  $f^k(u)=u,$ and that $\displaystyle\int_{\R^+}\Gamma(\tau)\d \tau=1,$ which is specifically required for \S\ref{NLL}.

We compute numerical solutions of \eqref{eq:kk} on the domain
 $[-5, \, 5]$ and the time
interval $[0, \, 0.5]$ with
\begin{align}
    \label{eq:ex1} U^1_0(x)=0.25\mathbbm{1}_{(-2,2)}(x), \quad &
    U^2_0(x)=\mathbbm{1}_{(-2,2)}(x).
    \end{align}\\
    
Figure \ref{fig:ex211} displays the numerical approximations of \eqref{eq:kk}-\eqref{eq:ex1} generated by the numerical  
scheme \eqref{scheme2}, 
with $\Delta x =0.00625, \eta=0.25$ and  $\delta=0.0125$. It can be seen that the numerical scheme is able to capture both shocks and rarefactions well. 
\begin{figure}[h!]
 \centering
\includegraphics[width=\textwidth,keepaspectratio]{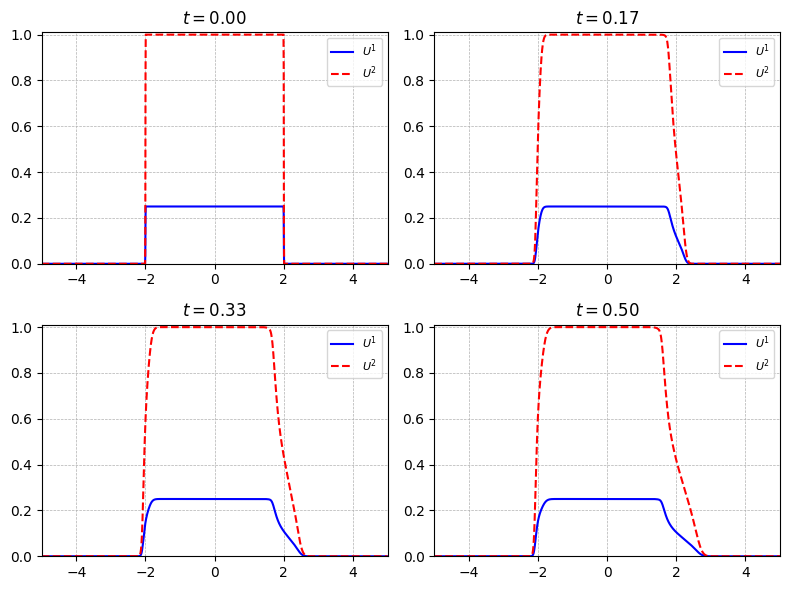}
\hfill
\caption{Solution to the nonlocal conservation law~\eqref{eq:kk}-\eqref{eq:ex1} on the domain $[-5,\,5]$ at times $t =
    0.00,\; 0.017,\;0.33, \: 0.5$, with mesh size $\Delta x=0.00625$. $U^1$({\full}),\,\,$U^2$({\color{red}\dashed}).}
  \label{fig:ex211}
\end{figure}

 \begin{figure}[h!]
  \centering \noindent\begin{minipage}{0.46\textwidth}
    \centering
    \begin{tabular}{|c|c|c|c|c|c|c|c|c|c|}\hline
     \multicolumn{1}{|c|}{ $ \delta$}&\multicolumn{1}{|c|}{$\frac{e_{\delta}(T)}{100}$}\vline & \multicolumn{1}{|c|}{$\alpha$}\vline\\
     \hline
     $4/5$&$31.11$&\tabularnewline
     \hline
     $2/5$&$19.01$&$0.71$\tabularnewline
     \hline
     $1/5$&$9.63$&$0.98$\tabularnewline
     \hline
     $1/10$&$4.83$&$0.99$\tabularnewline
     \hline
     $1/20$&$2.41$&$1.00$\tabularnewline
     \hline
     $1/40$&$1.17$&$1.04$\tabularnewline
     \hline
     $1/80$&$0.604$&$0.96$\tabularnewline
     \hline
 \end{tabular}
  \end{minipage}
\noindent\begin{minipage}{0.4\textwidth}
\includegraphics[width=\textwidth, trim = 40 25 20 5]{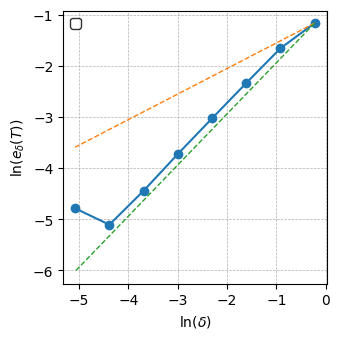}
  \end{minipage}  \caption{Convergence rate $\alpha$ for the memory-to-memoryless dynamics of solution $\boldsymbol{U}_{\delta}(T)$ with decreasing $\delta$:  Domain $[-5,\,5]$ at time $T=0.5$ for the problem~\eqref{eq:kk}-\eqref{eq:ex1}. Observed convergence rate({\color{blue}\oline}), theoretical convergence rate ({orange}), and reference line of slope $1$({\color{green}\dashed}).}\label{fig:ex21}
\end{figure}\begin{figure}[ht!]
 \centering
\includegraphics[width=\textwidth,keepaspectratio]{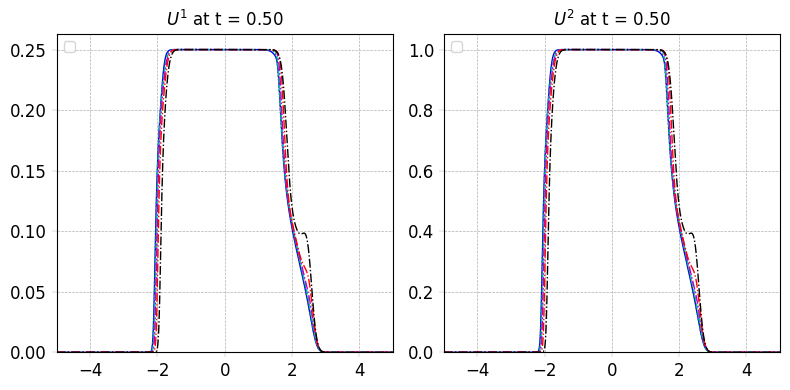}
\caption{Domain $[-5,\,5],   T=0.5,\Delta x =0.00625,\eta=0.25$: Solution to the ~\eqref{eq:kk}-\eqref{eq:ex1} with decreasing time convolution radii $\delta=$ $0.8$({\color{black}\chainn}),$0.4$({\color{magenta}\chainn}),$0.2$({\color{magenta}\chainn}), $0.1$({\color{cyan}\dashed}), $0.0125$({\color{green}\dotted}); Solution to the nonlocal-space only counterpart of ~\eqref{eq:kk}, with initial data \eqref{eq:ex1}({\color{blue}\full}).}
  \label{fig:ex22}
\end{figure}
 Figures \ref{fig:ex21}--\ref{fig:ex22} illustrate the memory-to-memoryless dynamics of \eqref{eq:kk}--\eqref{eq:ex1} as captured by the numerical scheme. To this end, we compute numerical approximations of \eqref{eq:kk}--\eqref{eq:ex1} using \eqref{scheme2} with $\Delta x = 0.00625$ and $\eta = 0.25$, while successively decreasing the temporal convolution radius $\delta$, starting from $\delta = 0.0125$ and halving it at each step. As shown in Figure \ref{fig:ex21}, the entropy solutions of the nonlocal space--nonlocal time system \eqref{eq:kk} converge to the entropy solution of the corresponding nonlocal-space-only system as the radius of the temporal convolution kernel $\Gamma$ tends to zero.

Let $\boldsymbol{U}_{\delta}(T,\cdot)$ denote the numerical solution at time $T$ corresponding to the temporal convolution radius $\delta$, computed using \eqref{scheme2} for the nonlocal space--nonlocal time system \eqref{eq:kk}--\eqref{eq:ex1} and let $\boldsymbol{U}(T,\cdot)$ of the nonlocal-space-only counterpart of \eqref{eq:kk}. In addition, we estimate the convergence rate of the memory-to-memoryless dynamics at time $T = 0.5$ 
 by measuring $
e_{\delta}(T) = \norma{\boldsymbol{U}_{\delta}(T,\cdot) - \boldsymbol{U}(T,\cdot)}_{(L^1(\mathbb{R}))^N}.$
The observed convergence rate is given by $
\alpha = \displaystyle\log_{2}\!\left(\frac{e_{\delta}(T)}{e_{\delta/2}(T)}\right),$
and is reported in Figure \ref{fig:ex22}. We observe that $\alpha > 0.5$, exceeding the theoretical rate established in Theorem~\ref{con}.

Figures \ref{fig:ex31}–\ref{fig:ex32} demonstrate the asymptotic compatibility of the scheme \eqref{scheme2} in the memory-to-memoryless limit. Numerical approximations of \eqref{eq:kk}–\eqref{eq:ex1} are computed with $\eta = 0.25$ using \eqref{scheme2}, \begin{figure}[h!]
  \centering \noindent\begin{minipage}{0.46\textwidth}
    \centering
    \begin{tabular}{|c|c|c|c|c|c|c|c|c|c|}\hline
     \multicolumn{1}{|c|}{ $ \frac{\Delta x}{0.0125}$}&\multicolumn{1}{|c|}{$e_{\Delta}(T)$}\vline & \multicolumn{1}{|c|}{$\alpha$}\vline\\
     \hline
      $1$&$0.63$&\tabularnewline
     \hline
     $1/2$&$0.38$&$0.72$\tabularnewline
     \hline
     $1/4$&$0.19$&$0.99$\tabularnewline
     \hline
     $1/8$&$0.096$&$1.00$\tabularnewline
    \hline
 \end{tabular}
  \end{minipage}
\noindent\begin{minipage}{0.4\textwidth}
\includegraphics[width=\textwidth, trim = 40 25 20 5]{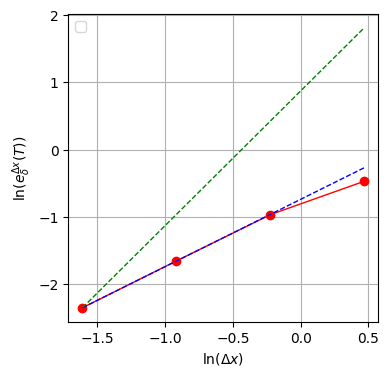}
  \end{minipage}  
  \vspace{1cm}\caption{Convergence rate $\alpha$ for the asymptotic compatibility of the scheme \eqref{scheme2} for memory-to-memoryless dynamics:  Domain $[-5,\,5]$ at time $T=0.5$ for the problem~\eqref{eq:kk}-\eqref{eq:ex1}. Observed convergence rate({\color{blue}\oline}), theoretical convergence rate ({orange}), and reference line of slope $1$({\color{green}\dashed}).}\label{fig:ex31}
\end{figure}\begin{figure}[ht!]
 \centering
\includegraphics[width=\textwidth,keepaspectratio]{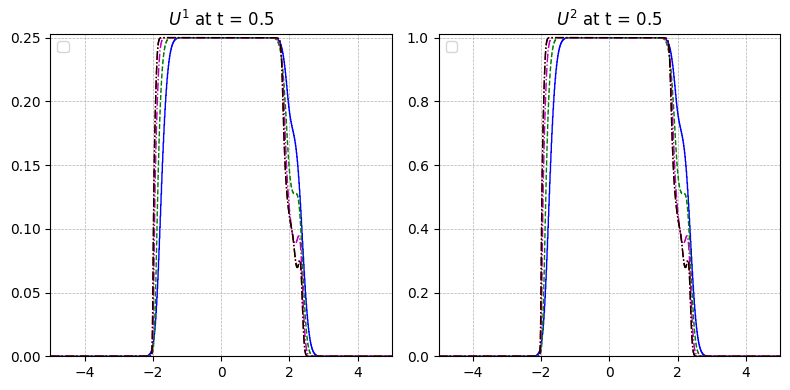}
\caption{Domain $[-5,\,5],   T=0.5,\Delta x =0.00625,\eta=0.25$, Solution to ~\eqref{eq:kk}-\eqref{eq:ex1} with decreasing $\Delta x$.
$\frac{\Delta x}{0.0125}=$ $1$({\color{black}\chainn}),$1/2$({\color{magenta}\chainn}),$1/4$({\color{magenta}\chainn}), $1/8$({\color{green}\dashed}); Solution to the nonlocal-space only counterpart of ~\eqref{eq:kk}, with initial data \eqref{eq:ex1}({\color{blue}\full}).}
  \label{fig:ex32}
\end{figure}while the spatial grid is successively refined. Starting from $\Delta x = 0.0125$, the mesh size is halved at each refinement step. Simultaneously, the temporal convolution radius $\delta$ is reduced in such a way that the ratio $\displaystyle \frac{\delta}{\Delta x} = 128$ remains fixed. As evidenced in Figure \ref{fig:ex31}, the entropy solutions of the nonlocal space–nonlocal time system \eqref{eq:kk} converge to the entropy solution of the corresponding nonlocal-space-only system in the limit $\delta \to 0$. Let $\boldsymbol{U}_{\delta}^{\Delta}(T,\cdot)$ denote the numerical solution at time $T$, corresponding to temporal convolution radius $\delta$ and spatial mesh size $\Delta x$, computed using \eqref{scheme2} for the nonlocal space--nonlocal time system \eqref{eq:kk}--\eqref{eq:ex1}. 
Let $\boldsymbol{U}^{\Delta_{\text{fine}}}(T,\cdot)$ denote the numerical solution at time $T$ of the corresponding nonlocal-space-only system, computed using \eqref{scheme2} on the finest spatial grid with mesh size $\Delta_{\text{fine}} = 0.0125/4$, which is taken as the reference solution. 
In addition, the convergence rate at time $T = 0.5$ is estimated by computing the $L^1$-error with respect to this reference solution. Specifically, we define $
e_{\Delta}(T) = \norma{\boldsymbol{U}_{\delta}^{\Delta}(T,\cdot) - \boldsymbol{U}^{\Delta_{\text{fine}}}(T,\cdot)}_{(L^1(\mathbb{R}))^N}.$
The observed convergence rate is then given by $
\alpha = \log_{2}\!\left(\frac{e_{\Delta}(T)}{e_{\Delta/2}(T)}\right)$. The computed rates are reported in Figure \ref{fig:ex32}. We observe that $\alpha > 0.5$, which exceeds the theoretical convergence rate established in Theorem~\ref{rate}.
\section{Conclusion and Future Directions}
In this work, using convergent finite volume approximations together with the continuous dependence estimates for scalar conservation laws with smooth coefficients, we prove well-posedness of the initial value problem \eqref{nlm}--\eqref{init}. As the temporal kernel concentrates, the entropy solutions of \eqref{nlm}--\eqref{init} converge to those of the purely spatially nonlocal system \eqref{nls}, \eqref{init}, with rate $\mathcal{O}(\sqrt{\delta})$, thereby justifying the memory-to-memoryless limit. For temporal kernels with a finite first moment, the finite volume schemes for \eqref{nlm}--\eqref{init} are shown to be compatible with this asymptotic behavior, with an asymptotic rate of $\mathcal{O}(\sqrt{\delta}+\sqrt{\Delta x})$. To the best of our knowledge, these are the first results addressing well-posedness and asymptotic analysis for such general systems with memory.

Stability estimates, error analysis of the proposed finite volume method, and extensions to less regular interaction kernels in the presence of source terms will be presented in forthcoming companion papers.  For the case of  $N=1$ and linear $f$, we also plan to investigate the asymptotic compatibility of the finite volume schemes \eqref{scheme2} for the nonlocal-to-local transition from \eqref{nlm}–\eqref{init} to \eqref{local}, \eqref{init}. Nevertheless, several challenging problems remain open:
\begin{itemize}
\item When nonlocality acts only in time (without spatial convolution) and the flux $f^k$ is nonlinear in \eqref{nlm}--\eqref{init}, the entropy framework and the spatial $\bv$ bounds do not directly carry over, since the spatial regularity of the nonlocal term is lost, and the system behaves like a local conservation law with discontinuous flux. Consequently, a general well-posedness theory for such PDEs, as well as the corresponding nonlocal-to-local asymptotic analysis, remain {\color{black}a} largely open problem at the time of writing, even for $N=1$.
{\color{black}However, for $N=1$, and $f$ linear, we believe that one could possibly prove the existence and uniqueness of the weak solutions (without additional entropy condition)} using fixed-point arguments.
\item With $N>1$, the complete passage from  \eqref{nlm}-\eqref{init} to fully local system \eqref{local}, \eqref{init} for general nonlinear fluxes is also largely open at this time, as the {\color{black}well-posedness theory for }local system, \eqref{local},\eqref{init} with $N>1$ is again underdeveloped for general $\bv$ data. {\color{black}Alternatively, one needs to curate a special class of nonlocal systems so that the corresponding local system is well-posed for $\bv$ initial data.}
\end{itemize}
\bmhead{Acknowledgements}This work was partially supported by AA's Seed Money Grant SM/08/2025-26 and the ARG Matrics Grant ANRF/ARGM/2025/001976/MTR from the Anusandhan National Research Foundation (ANRF), India, as well as GV's INSPIRE Faculty Fellowship (IFA24-MA215) from the Department of Science and Technology (DST), Government of India. The authors also acknowledge the hospitality of the Department of Mathematics, Penn State University, where part of this work was completed during AA's visit and GV's postdoctoral tenure.
\bibliographystyle{siam}
\bibliography{references}

@preamble{
   "\def\cprime{$'$} "
}

@article{BHL2023,
  title={A {H}illiges-{W}eidlich-type scheme for a one-dimensional scalar conservation law with nonlocal flux.},
  author={B{\"u}rger, Raimund and Contreras, Harold Deivi and Villada, Luis Miguel},
  journal={Netw. Heterog. Media},
  volume={18},
  number={2},
  pages={664--693},
  year={2023},
}

@article{AHV2024,
year = {2025},
month = {oct},
publisher = {IOP Publishing},
volume = {38},
number = {10},
pages = {105007},
author = {Aggarwal, Aekta and Holden, Helge and Vaidya, Ganesh},
title = {Error estimates for systems of nonlocal balance laws modelling dense multilane vehicular traffic},
journal = {Nonlinearity},
}

@article{BMV2025,
  title={Uniqueness Domains for ${L}^{\infty}$ Solutions of $2\times2$ Hyperbolic Conservation Laws},
  author={Bressan, Alberto and Marconi, Elio and Vaidya, Ganesh},
  journal={Arch. Ration. Mech. Anal.},
  volume={249},
  number={6},
  pages={1--37},
  year={2025},
  publisher={Springer}
}

@article{CG2019,
  title={Non-local multi-class traffic flow models},
  author={Chiarello, Felisia Angela and Goatin, Paola},
  journal={ Netw. Heterog. Media},
  volume={14},
  number={2},
  pages={371--387},
  year={2019}
}

@article{KLS2018,
  title={Analysis of a system of nonlocal balance laws with weighted work in progress},
  author={Keimer, Alexander and Leugering, G{\"u}nter and Sarkar, Tanmay},
  journal={J. Hyperbolic Differ. Eq.},
  volume={15},
  number={03},
  pages={375--406},
  year={2018},
  publisher={World Scientific}
}

@article{CC2007,
  title={Solutions for a nonlocal conservation law with fading memory},
  author={Chen, Gui-Qiang and Christoforou, Cleopatra},
  journal={Proc. Am. Math. Soc.},
  volume={135},
  number={12},
  pages={3905--3915},
  year={2007}
}

@inproceedings{c2008,
  title={Nonlocal conservation laws with memory},
  author={Christoforou, C},
  booktitle={Hyperbolic problems: theory, numerics, applications},
  pages={381--388},
  year={2008},
  organization={Springer}
}

@article{DAF2012,
  title={Development of singularities in the motion of materials with fading memory},
  author={Dafermos, C M },
  journal={Arch. Rational Mech. Anal.},
    volume={91},
  pages={193--205},
  year={1986},
  publisher={Springer}
}

@article{CHR2007,
  title={Systems of hyperbolic conservation laws with memory},
  author={Christoforou, C},
  journal={J. Hyper. Diff. Eqs.},
  volume={4},
  number={03},
  pages={435--478},
  year={2007},
  publisher={World Scientific}
}

@article{ACG2015,
  title={Nonlocal systems of conservation laws in several space dimensions},
  author={Aggarwal, Aekta and Colombo, Rinaldo M. and Goatin, Paola},
  journal={SIAM J. Numer. Anal.},
  volume={53},
  number={2},
  pages={963--983},
  year={2015},
  publisher={SIAM}
}

@article{AG2016,
  title={Crowd dynamics through non-local conservation laws},
  author={Aggarwal, Aekta and Goatin, Paola},
  journal={Bull. Braz. Math. Soc. (N.S.)},
  volume={47},
  number={1},
  pages={37--50},
  year={2016},
  publisher={Springer}
}

@article{AHV2023_1,
    author = {Aggarwal, Aekta and Holden, Helge and Vaidya, Ganesh},
    title = {Well-posedness and error estimates for coupled systems of nonlocal conservation laws},
    journal = {IMA J. Numer. Anal.},
    volume = {44},
    number = {6},
    pages = {3354-3392},
    year = {2024},
    month = {01},
    issn = {0272-4979},
    doi = {10.1093/imanum/drad101},
    url = {https://doi.org/10.1093/imanum/drad101},
    eprint = {https://academic.oup.com/imajna/article-pdf/44/6/3354/60925381/drad101.pdf},
}

@article{AHV2023,
  title={On the accuracy of the finite volume approximations to nonlocal conservation laws},
  author={Aggarwal, Aekta and Holden, Helge and Vaidya, Ganesh},
  journal={Numer. Math.},
  volume={156},
  number={1},
  pages={237--271},
  year={2024},
  publisher={Springer}
}

@article{AV2023,
  title={Convergence of the numerical approximations and well-posedness: Nonlocal conservation laws with rough flux},
  author={Aggarwal, Aekta and Vaidya, Ganesh},
  journal={Math. Comp.},
  volume={94},
  number={352},
  pages={585--610},
  year={2025}
}

@article{ANT2007,
  title={First-order continuous models of opinion formation},
  author={Aletti, Giacomo and Naldi, Giovanni and Toscani, Giuseppe},
  journal={SIAM J. Appl. Math.},
  volume={67},
  number={3},
  pages={837--853},
  year={2007},
  publisher={SIAM}
}

@article {AS2012,
  author ={Amadori, Debora  and Shen, Wen},
  title ={An integro-differential conservation law arising in
                  a model of granular flow},
  journal ={J. Hyperbolic Differ. Equ.},
  Fjournal =	 {J. Hyperbolic Differ. Equ.},
  volume =	 9,
  year =	 2012,
  NUMBER =	 1,
  PAGES =	 {105--131},
  ISSN =	 {0219-8916},
  MRCLASS =	 {35L65 (35L04 35R09)},
  MRNUMBER =	 2910979,
  MRREVIEWER =	 {Hermano Frid},
  DOI =		 {10.1142/S0219891612500038},
  URL =		 {http://dx.doi.org/10.1142/S0219891612500038},
}

@article{ACT2015,
  title={On the numerical integration of scalar nonlocal conservation laws},
  author={Amorim, Paulo and Colombo, Rinaldo M and Teixeira, Andreia},
  journal={ESAIM Math. Model. Numer. Anal.},
  volume={49},
  number={1},
  pages={19--37},
  year={2015},
  publisher={EDP Sciences}
}

@article{BBKT2011,
  title={On nonlocal conservation laws modelling sedimentation},
  author={Betancourt, Fernando and B{\"u}rger, Raimund and Karlsen, Kenneth H and Tory, Elmer M},
  journal={Nonlinearity},
  volume={24},
  number={3},
  pages={855},
  year={2011},
  publisher={IOP Publishing}
}

@article{BG2016,
  title={Well-posedness of a conservation law with non-local flux arising in traffic flow modeling},
  author={Blandin, Sebastien and Goatin, Paola},
  journal={Numer. Math.},
  volume={132},
  number={2},
  pages={217--241},
  year={2016},
  publisher={Springer}
}

@article{CG2023,
  title={{A non-local system modeling bi-directional traffic flows}},
  author = {Chiarello, Felisia Angela and Goatin, Paola},
  URL = {https://hal.science/hal-03722225},
  journal = {SEMA SIMAI Springer Ser.},
  publisher = {Springer International Publishing},
  volume={32},
  year={2023},
  HAL_ID = {hal-03722225},
  HAL_VERSION = {v1},
}

@article{FCV2023,
  title = {Existence of Entropy Weak Solutions for {1D}  Non-Local Traffic Models with Space Discontinuous Flux},
  author = {{Felisia Angela Chiarello} and Contreras, Harold Devi and Villada, Luis Miguel},
  year = {2023},
  journal = {J. Engrg. Math.}
}

@article {CGL2012,
  author ={Colombo, R. M. and Garavello, M. and L{\'e}cureux-Mercier, M.},
  title =	{A class of nonlocal models for pedestrian traffic},
  journal ={Math. Mod. Met. Appl. Sci.},
  volume = 22,
  year = 2012,
  NUMBER = 4,
  PAGES =	{1150023-34},
  ISSN = {0218-2025},
  MRCLASS = {90B20 (35Qxx)},
  MRNUMBER = 2902155,
  DOI = {10.1142/S0218202511500230},
  URL =		 {http://dx.doi.org/10.1142/S0218202511500230},
}

@article{CM2015,
  title={Nonlocal systems of balance laws in several space dimensions with applications to laser technology},
  author={Colombo, Rinaldo M and Marcellini, Francesca},
  journal={J. Differential Equations },
  volume={259},
  number={11},
  pages={6749--6773},
  year={2015},
  publisher={Elsevier}
}

@article{CK24,
  author = {F. A. Chiarello and A. Keimer},
  title = {On the singular limit problem in nonlocal balance laws: applications to nonlocal lane-changing traffic flow models},
  journal = {J. Math. Anal. Appl.},
  volume = {537},
  number = {2},
  year = {2024},
  pages = {128358}
}

@article{KP2019,
  title={Nonlocal conservation laws with time delay},
  author={Keimer, Alexander and Pflug, Lukas},
  journal={NoDEA Nonlinear Differential Equations Appl.},
  volume={26},
  number={6},
  pages={54},
  year={2019},
  publisher={Springer}
}

@article{BFK2022,
  title={Modeling multilane traffic with moving obstacles by nonlocal balance laws},
  author={Bayen, Alexandre and Friedrich, Jan and Keimer, Alexander and Pflug, Lukas and Veeravalli, Tanya},
  journal={SIAM J. Appl. Dyn. Syst.},
  volume={21},
  number={2},
  pages={1495--1538},
  year={2022},
  publisher={SIAM}
}

@article{FGKP2022,
author = {Friedrich, Jan and G\"{o}ttlich, Simone and Keimer, Alexander and Pflug, Lukas},
title = {Conservation Laws with Nonlocal Velocity: The Singular Limit Problem},
journal = {SIAM J. Appl. Math.},
volume = {84},
number = {2},
pages = {497-522},
year = {2024},
}

@article{FGR2021,
  title={Nonlocal approaches for multilane traffic models},
  author={Friedrich, Jan and G{\"o}ttlich, Simone and Rossi, Elena},
  journal={Commun. Math. Sci.},
  volume={19},
  number={8},
  pages={2291--2317},
  year={2021},
  publisher={International Press of Boston}
}

@article {GHS+2014,
   title={Modeling, simulation and validation of material flow on conveyor belts},
  author={G{\"o}ttlich, Simone and Hoher, Simon and Schindler, Patrick and Schleper, Veronika and Verl, Alexander},
  journal={Appl. Math. Model.},
  volume={38},
  number={13},
  pages={3295--3313},
  year={2014},
  publisher={Elsevier}
}

@book{HR2015,
  title={Front {T}racking for {H}yperbolic {C}onservation {L}aws},
  author={Holden, Helge and Risebro, Nils Henrik},
  year={2015},
edition="Second",
  publisher={Springer}
}

@article{DHSS2023,
  title={A space-time nonlocal traffic flow model: Relaxation representation and local limit},
  author={Du, Qiang and Huang, Kuang and Scott, James and Shen, Wen},
  journal={Discrete Contin. Dyn. Syst.},
  volume={43},
  number={9},
  year={2023},
  publisher={AIMS}
}

@article{KP2021,
  title={Discontinuous nonlocal conservation laws and related discontinuous ODEs--Existence, Uniqueness, Stability and Regularity},
  author={Keimer, Alexander and Pflug, Lukas},
  journal={Comptes Rendus. Math{\'e}matique},
  volume={361},
  number={G11},
  pages={1723--1760},
  year={2023}
}

@article{BS2021,
   title={Entropy Admissibility Of The Limit Solution For A Nonlocal Model Of Traffic Flow},
  author={Bressan, Alberto and Shen, Wen},
  journal={Commun. Math. Sci.},
  volume={19},
  number={5},
  pages={1447--1450},
  year={2021},
  publisher={International Press of Boston, Inc.}
}

@article{CCS2019,
  title={On the singular local limit for conservation laws with nonlocal fluxes},
  author={Colombo, Maria and Crippa, Gianluca and Spinolo, Laura V},
  journal={Arch. Ration. Mech. Anal.},
  volume={233},
  number={3},
  pages={1131--1167},
  year={2019},
  publisher={Springer}
}

@article{CGES2021,
  title={Local limit of nonlocal traffic models: convergence results and total variation blow-up},
  author={Colombo, Maria and Crippa, Gianluca and Marconi, Elio and Spinolo, Laura V},
 journal={Ann. Inst. H. Poincar{\'e} C Anal. Non Lin{\'e}aire},
  volume={38},
  number={5},
  pages={1653--1666},
  year={2021},
  organization={Elsevier}
}

@article{GNAL2021,
title = {Singular limits with vanishing viscosity for nonlocal conservation laws},
journal = {Nonlinear Anal.},
volume = {211},
pages = {112370},
year = {2021},
issn = {0362-546X},
doi = {https://doi.org/10.1016/j.na.2021.112370},
url = {https://www.sciencedirect.com/science/article/pii/S0362546X21000808},
author = {Giuseppe Maria Coclite and Nicola {De Nitti} and Alexander Keimer and Lukas Pflug}
}

@article{KR2003,
  title={On the uniqueness and stability of entropy solutions of nonlinear degenerate parabolic equations with rough coefficients},
  author={Karlsen, Kenneth Hvistendahlvistendahl and Risebro, Nils Henrik},
  journal={Discrete Contin. Dyn. Syst.},
  volume={9},
  number={5},
  pages={1081},
  year={2003},
  publisher={American Institute of Mathematical Sciences}
}

@article{CAL2020,
   author = {Coron, Jean-Michel and Alexander Keimer and Lukas Pflug},
   doi = {10.1137/20M1331652},
   issn = {10957154},
   issue = {6},
   journal = {SIAM J. Math. Anal.},
   title = {Nonlocal Transport Equations-Existence and Uniqueness of Solutions and Relation to the Corresponding Conservation Laws},
   volume = {52},
number = {6},
pages = {5500-5532},
   year = {2020},
}

@book{B2000,
  title={Hyperbolic systems of conservation laws: the one-dimensional Cauchy problem},
  author={Bressan, Alberto},
  volume={20},
  year={2000},
  publisher={OUP Oxford}
}

@article{liu2020,
  title={Long-time behavior of a class of viscoelastic plate equations},
  author={Liu, Yang},
  journal={Electron. Res. Arch.},
  volume={28},
  number={1},
  pages={311},
  year={2020},
  publisher={American Institute of Mathematical Sciences}
}

@article{Dafermos1970,
  author    = {Dafermos, C M },
  title     = {Asymptotic stability in viscoelasticity},
  journal   = {Arch. Ration. Mech. Anal.},
  volume    = {37},
  year      = {1970},
  pages     = {297--308}
}

@book{Podlubny1999,
  author    = {Igor Podlubny},
  title     = {Fractional Differential Equations},
  publisher = {Academic Press},
  address   = {San Diego},
  year      = {1999}
}

@article{Clarkson1999,
  author = {Clarkson, C. R. and Bustin, R. M.},
  title = {The effect of pore structure and gas pressure upon the transport properties of coal: a laboratory and modeling study. 1. Isotherms and pore volume distributions},
  journal = {Fuel},
  volume = {78},
  pages = {1333--1344},
  year = {1999}
}

@article{Shi2003,
  author = {Shi, J. Q. and Durucan, S.},
  title = {A bidisperse pore diffusion model for methane displacement desorption in coal by CO$_2$ injection},
  journal = {Fuel},
  volume = {82},
  pages = {1219--1229},
  year = {2003},
  doi = {10.1016/S0016-2361(03)00049-7}
}

@article{Haggerty1995,
  author = {Haggerty, R. and Gorelick, S. M.},
  title = {Multiple-rate mass transfer for modeling diffusion and surface reactions in media with pore-scale heterogeneity},
  journal = {Water Resour. Res.},
  volume = {31},
  pages = {2383--2400},
  year = {1995}
}

@article{Gooseff2003,
  author = {Gooseff, M. N. and Wondzell, S. M. and Haggerty, R. and Anderson, J.},
  title = {Comparing transient storage modeling and residence time distribution (RTD) analysis in geomorphically varied reaches in the Lookout Creek basin, Oregon, USA},
  journal = {Adv. Water Resour.},
  volume = {26},
  pages = {925--937},
  year = {2003}
}

@article{Haggerty2002,
  author = {Haggerty, R. and Wondzell, S. M. and Johnson, M. A.},
  title = {Power-law residence time distribution in the hyporheic zone of a 2nd-order mountain stream},
  journal = {Geophys. Res. Lett.},
  volume = {29},
  number = {13},
  pages = {1811},
  year = {2002},
  doi = {10.1029/2002GL014743}
}

@article{P2014,
  title={Numerical scheme for a scalar conservation law with memory},
  author={ Peszy{\'n}ska, Ma{\l}gorzata},
  journal={Numer. Methods Partial Differential Equations},
  volume={30},
  number={1},
  pages={239--264},
  year={2014},
  publisher={Wiley Online Library}
}

@incollection{N2023,
  title={A nonlinear conservation law with memory},
  author={Nohel, JA},
  booktitle={Volterra and Functional Differential Equations},
  pages={91--123},
  year={2023},
  publisher={CRC Press}
}

@incollection{D1987,
  title={Solutions with shocks for conservation laws with memory},
  author={Dafermos, C M },
  booktitle={Amorphous Polymers and Non-Newtonian Fluids},
  pages={33--55},
  year={1987},
  publisher={Springer}
}

@article{CCDKP2022,
   title={A general result on the approximation of local conservation laws by nonlocal conservation laws: The singular limit problem for exponential kernels},
   author={Coclite, Giuseppe Maria and Coron, Jean-Michel and De Nitti, Nicola and Keimer, Alexander and Pflug, Lukas},
   journal={Ann. Inst. H. Poincar{\'e} C Anal. Non Lin{\'e}aire},
  year={2022},
  publisher={EDP Sciences},
 }

@article{BS2020,
  title={On traffic flow with nonlocal flux: a relaxation representation},
  author={Bressan, Alberto and Shen, Wen},
  journal={Arch. Ration. Mech. Anal.},
  volume={237},
  number={3},
  pages={1213--1236},
  year={2020},
  publisher={Springer}
}

@article{FKG2018,
  title={A {G}odunov type scheme for a class of {LWR} traffic flow models with non-local flux},
  author={Friedrich, Jan and Kolb, Oliver and G{\"o}ttlich, Simone},
  journal={Netw. Heterog. Media},
  volume={13},
  number={4},
  pages={531--547},
  year={2018},
  publisher={Netw. Heterog. Media}
}

@article{DH2024,
  title={Asymptotic compatibility of a class of numerical schemes for a nonlocal traffic flow model},
  author={Huang, Kuang and Du, Qiang},
  journal={SIAM J. Numer. Anal.},
  volume={62},
  number={3},
  pages={1119--1144},
  year={2024},
  publisher={SIAM}
}

@article{NH2025,
  title={Asymptotically compatible entropy-consistent discretization for a class of nonlocal conservation laws},
  author={De Nitti, Nicola and Huang, Kuang},
  journal={arXiv preprint arXiv:2510.00221},
  year={2025}
}

\end{document}